\documentclass[10pt]{amsart}

\usepackage{amsmath,amssymb,amsthm}
\usepackage{mathrsfs}
\usepackage{geometry}
\usepackage{enumitem}
\usepackage{mathtools}
\usepackage{stmaryrd}

\usepackage[dvipsnames]{xcolor}
\usepackage[
unicode=true,
colorlinks=true,
linktoc=all,
linkcolor=MidnightBlue,
citecolor=ForestGreen,
urlcolor=BrickRed,
filecolor=BrickRed,
bookmarksnumbered=true,
bookmarksopen=true
]{hyperref}

\hypersetup{
	pdftitle={Intrinsic Brown--York Type Mass at Infinity in Four Dimensions},
	pdfauthor={Jiangcheng You},
	pdfsubject={Differential Geometry, Geometric Analysis},
	pdfkeywords={
		Brown--York type mass,
		ADM mass,
		asymptotically flat manifolds,
		large hypersurfaces,
		quasilocal mass
	}
}

\newcommand{\R}{\mathbb{R}}

\newcommand{\Sph}{\mathbb{S}}
\newcommand{\gE}{g_{\mathrm{E}}}

\newcommand{\tr}{\operatorname{tr}}

\newcommand{\Ric}{\operatorname{Ric}}
\newcommand{\Rm}{\operatorname{Rm}}

\newcommand{\Weyl}{\operatorname{W}}

\newcommand{\Vol}{\operatorname{Vol}}
\newcommand{\Id}{\operatorname{Id}}

\newcommand{\Scal}{\operatorname{R}}
\newcommand{\Lie}{\mathcal{L}}

\newcommand{\diam}{\mathrm{diam}}
\newcommand{\End}{\mathrm{End}}

\newcommand{\ADM}{\mathrm{ADM}}

\newcommand{\SU}{\mathrm{SU}}

\newcommand{\diag}{\mathrm{diag}}
\newcommand{\BY}{\mathrm{BY}}

\newcommand{\mass}{\mathrm{mass}}

\theoremstyle{plain}
\newtheorem{theorem}{Theorem}[section]
\newtheorem{proposition}[theorem]{Proposition}
\newtheorem{lemma}[theorem]{Lemma}
\newtheorem{corollary}[theorem]{Corollary}

\newtheorem{definition}[theorem]{Definition}
\newtheorem{assumption}[theorem]{Assumption}

\theoremstyle{remark}
\newtheorem{remark}[theorem]{Remark}

\title{Intrinsic Brown--York Type Mass at Infinity in Four Dimensions}

\author{Jiangcheng You}
\address{
	School of Mathematical Sciences,
	University of Science and Technology of China,
	Hefei 230026, China
}
\email{yjcmp@mail.ustc.edu.cn}

\date{July 2, 2026}



\begin{document}

	\begin{abstract}
		We study a Brown--York type mass for closed hypersurfaces in
		four-dimensional asymptotically flat manifolds. The reference mean
		curvature is defined intrinsically as the trace of the positive
		solution of the contracted Gauss equation. For large uniformly convex
		hypersurfaces with controlled scale, we derive an expansion consisting
		of a boundary term converging to the ADM mass and a shape-dependent
		correction. For the four-dimensional analogue of the nearly round
		surfaces of Shi--Wang--Wu, this correction vanishes under a natural
		decay compatibility condition.
	\end{abstract}

	\maketitle

	\section{Introduction}
	
	Let $(M^4,g)$ be a smooth asymptotically flat
	Riemannian manifold. Following the standard formulations in
	\cite{Bartnik86,LeeParker,McCormick24}, we say that an end of $(M^4,g)$ is
	\emph{asymptotically flat of order $q$} if there exist a compact set
	$K\subset M$, a radius $R>0$, and a diffeomorphism
	\begin{align*}
		\Phi:M\setminus K\longrightarrow \R^4\setminus B_R
	\end{align*}
	such that, in the corresponding coordinates,
	\begin{align*}
		\partial^\beta(g_{ij}-\delta_{ij})
		=
		\mathcal O\bigl(|x|^{-q-|\beta|}\bigr),
		\qquad
		|\beta|\leq 2
	\end{align*}
	for some $q>1$. We fix such an asymptotically flat coordinate
	chart and identify the end with $\R^4\setminus B_R$. We also assume throughout that the scalar curvature of $M$ is $L^1$-integrable.
	
	The ADM mass, introduced by Arnowitt, Deser, and Misner \cite{ADM}, is
	defined in dimension four by
	\begin{align}\label{eq:four-dimensional-ADM-mass}
		m_{\ADM}(g)
		:=
		\frac{1}{6\omega_3}
		\lim_{r\to\infty}
		\int_{\Sph_r}
		\left(
		\partial_jg_{ij}
		-
		\partial_ig_{jj}
		\right)
		(\nu_E)^i\,dS_E,
	\end{align}
	where $\Sph_r=\{x\in\R^4:|x|=r\}$, $\nu_E$ and $dS_E$ denote the
	Euclidean outward unit normal and hypersurface measure, respectively, and
	$\omega_3=|\Sph^3|$. Under the assumptions above, this limit is finite and
	independent of the choice of asymptotically flat coordinates
	\cite{Bartnik86,McCormick24}.
	
	We next recall the classical Brown--York mass in dimension three. Let $(N^3,g)$ be a compact Riemannian
	$3$--manifold with boundary, and let $\Sigma$ be a connected boundary
	component with induced metric $\sigma=g|_\Sigma$. Suppose that $\Sigma$ is
	a topological $2$--sphere and that $\sigma$ has positive Gauss curvature.
	By Nirenberg's solution of the Weyl problem \cite{Nirenberg},
	$(\Sigma,\sigma)$ admits an isometric embedding into $\mathbb R^3$
	as a strictly convex closed surface. By the rigidity theorem for
	convex surfaces \cite{CohnVossen}, this embedding is unique up to
	rigid motions. Let $H_0$ be
	the mean curvature of this reference surface in $\mathbb R^3$, and let
	$H$ be the physics mean curvature of $\Sigma$ in $(N,g)$, both computed with
	respect to the outward unit normal. The Brown--York mass is defined by
	\begin{align}\label{eq:classical-BY}
		m_{\mathrm{BY}}(\Sigma)
		:=
		\frac{1}{8\pi}
		\int_{\Sigma}(H_0-H)\,d\mu_\sigma .
	\end{align}
	It was introduced by Brown and York through the Hamilton--Jacobi analysis
	of the gravitational action, and in the time-symmetric case it reduces to
	this Riemannian expression \cite{BrownYork}.
	
	Unfortunately, in higher dimensions the direct classical definition
	\eqref{eq:classical-BY} cannot be applied without modification, since the
	required flat reference embedding need not exist.  This issue already occurs
	for certain scalar-flat asymptotically flat ends arising from the
	trivial-group case of the refined ALE asymptotics in \cite{YouALE26}.  For
	such ends, the asymptotic geometry carries a Weyl tensor
	$\Weyl_\infty$ at infinity.  Using the local isometric-embedding criterion
	of Li--Weinstein \cite{LiWeinstein}, we show that, when
	$\Weyl_\infty\neq0$, the metrics induced on the large coordinate spheres
	$\Sph_r:=\{x\in\R^4:|x|=r\}$ are not even locally isometrically embeddable
	into $\R^4$ for all sufficiently large $r$.  See
	Subsection~\ref{subsec:local-embed} for details.
	
	We therefore replace the embedding-based reference term by an intrinsic one.
	The guiding observation is that, if a hypersurface isometrically embeds into a
	flat ambient manifold with shape operator $A$, then its induced metric $\sigma$
	and $A$ satisfy the contracted Gauss equation
	\begin{align*}
		\Ric_\sigma^\sharp
		=
		(\tr_\sigma A)A-A^2 .
	\end{align*}
	Thus the reference mean curvature can be obtained from the intrinsic
	boundary metric by solving this equation for a reference shape operator, rather
	than by first constructing a flat reference embedding. In the positive sectional
	curvature case this choice is canonical: setting
	$B_\sigma=\frac12\Scal_\sigma\Id-\Ric_\sigma^\sharp$, where $\Scal$ is the scalar curvature, the positive solution is
	\begin{align*}
		A_0^+[\sigma]
		=
		\sqrt{\det B_\sigma}\,B_\sigma^{-1},
	\end{align*}
	and we define $H_0^+[\sigma]:=\tr A_0^+[\sigma]$. This reference term is the one
	used in our Brown--York type mass. The viewpoint is consistent with the original
	Brown--York formulation \cite{BrownYork}, where the subtraction term fixing the
	zero level of quasilocal energy is determined by the prescribed boundary
	geometry. It is also close in spirit to Mann--Marolf's holographic
	renormalization of asymptotically flat spacetimes \cite{MannMarolf}.
	
	This gives the Brown--York type mass a new explanation. For a smooth closed
	hypersurface $\Sigma=\partial D\subset (N^4,g)$ whose induced metric
	$\sigma$ has positive sectional curvature, we take
	$H_0^+[\sigma]$ as the reference mean curvature and compare
	it with the physical mean curvature $H_g(\Sigma)$, computed with respect to
	the outward unit normal of $D$. We then set
	\begin{align}\label{eq:Brown-York-type-mass-introduction}
		m_{\BY}(\Sigma)
		:=
		\frac{1}{3\omega_3}
		\int_\Sigma
		\bigl(
		H_0^+[\sigma]-H_g(\Sigma)
		\bigr)\,d\mu_\sigma,
		\quad
		\omega_3=|\Sph^3|=2\pi^2 .
	\end{align}
	This agrees with the usual Brown--York reference term whenever a
	flat reference hypersurface exists. Its advantage is that it remains meaningful even when such an
	embedding is not available.
	
	\subsection{Main results and related works}
	
	After introducing the Brown--York type mass \eqref{eq:Brown-York-type-mass-introduction} in the four-dimensional AF setting, we address the
	following question: \textit{among the families of large closed hypersurfaces going to infinity, which ones recover the ADM mass through the Brown--York type mass?}
	
	This question is first of all a consistency test for the definition. A
	Brown--York type mass is attached to a finite boundary, but along large
	boundaries tending to an asymptotic end it should recover the corresponding
	conserved quantity at infinity. In the asymptotically flat Riemannian setting
	this quantity is the ADM mass; for the classical Brown--York mass, the
	large-sphere limit along coordinate spheres was proved by Fan--Shi--Tam
	\cite{FanShiTam}. This is consistent with the Hamiltonian viewpoint, where the
	conserved quantity at spatial infinity is the ADM energy-momentum
	\cite{ADM,Szabados}, and also with the spatial-infinity limit in the
	Wang--Yau theory \cite{WangYau,WangYau09}. In our geometric setting, the corresponding
	total quantity is the ADM mass defined as \eqref{eq:four-dimensional-ADM-mass}. Thus, after rewriting $m_{\BY}$ by \eqref{eq:Brown-York-type-mass-introduction}, one must check
	whether it recovers this mass along suitable large hypersurfaces.
	
	The question also has a geometric aspect: the way in which the boundary
	tends to infinity can affect the limiting behavior. The ADM limit is not merely
	a statement about coordinate spheres. In asymptotically flat $3$--manifolds,
	Shi--Wang--Wu proved convergence of the Brown--York mass and the Hawking mass
	along nearly round surfaces \cite{ShiWangWu}, while Fan--Kwong proved
	convergence of the Brown--York mass along certain non-round convex revolution
	surfaces in asymptotically Schwarzschild manifolds
	\cite{FanKwong11,FanKwong13}. Related constructions, including inverse mean
	curvature flow, the Huisken--Yau CMC foliation, isoperimetric mass, and
	Bartnik's quasilocal mass, also show that large boundaries or large regions can
	encode information about the total mass
	\cite{HuiskenIlmanen,HuiskenYau,JaureguiLee,Bartnik89,McCormick24}.
	
	We now formulate our large-boundary result in this direction. We work in
	a fixed AF coordinate system, and consider families of large hypersurfaces, denoted by $\{\Sigma_a\}$, whose rescaled
	geometry is uniformly controlled; see Assumption~\ref{ass:large-hypersurfaces}. The point is to go beyond coordinate spheres while retaining enough control to compare the Brown--York type mass
	with the asymptotic data of the end.
	
	For such a family, the Brown--York type mass has an asymptotic expansion whose main term converges to the ADM mass. The possible obstruction to convergence is a shape-dependent correction obtained by pairing the Euclidean geometry of $\Sigma_a$ with the asymptotic metric perturbation $h=g-g_E$. Thus, the Brown--York type mass recovers the ADM mass along the family precisely when this obstruction converges to zero. Our first main result can be stated as
	
	\begin{theorem}[Large-boundary expansion of the Brown--York type mass]
		\label{thm:BY-AF-large-boundary}
		Let $(M^4,g)$ have an asymptotically flat end of order $q>1$, and assume that
		$R_g\in L^1(M,dV_g)$. Fix an asymptotically flat coordinate chart, and let
		$\{\Sigma_a\}$ be a family of closed hypersurfaces satisfying
		Assumption~\ref{ass:large-hypersurfaces}. Set $h:=g-\gE$. Then, for all sufficiently large $a$, the Brown--York type mass
		$m_{\BY}(\Sigma_a)$ is well-defined and satisfies
		\begin{align}\label{eq:BY-AF-main-expansion}
			m_{\BY}(\Sigma_a)
			=
			\frac{1}{6\omega_3}
			\int_{\Sigma_a}
			\bigl(
			\partial_jh_{ij}
			-
			\partial_ih_{jj}
			\bigr)(\nu_E)^i\,d\mu_E
			+
			\frac{1}{3\omega_3}\int_{\Sigma_a}
			\mathfrak D_a^{\alpha\beta}
			h_{\alpha\beta}\,d\mu_E
			+
			\mathcal O(\rho_a^{2-2q}).
		\end{align}
		Here, $\mathfrak D_a$, defined as \eqref{eq:BY-def-EP} and \eqref{eq:BY-def-D}, depends only on the Euclidean geometry of $\Sigma_a$ and satisfies
		\begin{align*}
			\nabla_\alpha\mathfrak D_a^{\alpha\beta}=0,
			\qquad
			(K_a^E)_{\alpha\beta}\mathfrak D_a^{\alpha\beta}=0.
		\end{align*}
		Consequently,
		\begin{align}\label{eq:BY-AF-main-criterion}
			\lim_{a\to\infty}m_{\BY}(\Sigma_a)
			=
			m_{\ADM}(g)
			\quad\Longleftrightarrow\quad
			\int_{\Sigma_a}
			\mathfrak D_a^{\alpha\beta}
			h_{\alpha\beta}\,d\mu_E
			\longrightarrow 0.
		\end{align}
	\end{theorem}
	
	\begin{remark}
		The main point of Theorem~\ref{thm:BY-AF-large-boundary} is the
		criterion \eqref{eq:BY-AF-main-criterion}.  For large coordinate
		spheres, this criterion is satisfied; see
		Corollary~\ref{cor:BY-AF-coordinate-spheres}.  This may be viewed as
		a four-dimensional analogue of the large-sphere limit for the
		classical Brown--York mass in asymptotically flat three-manifolds
		proved by Fan--Shi--Tam \cite{FanShiTam}.
	\end{remark}
	
	We next apply the criterion
	\eqref{eq:BY-AF-main-criterion} to a natural four-dimensional analogue
	of the nearly round surfaces introduced by Shi--Wang--Wu
	\cite{ShiWangWu}.  In the three-dimensional asymptotically flat setting,
	their notion describes large surfaces that become asymptotically umbilic
	at their natural distance scale, while their radii, diameter, and area
	remain uniformly controlled at that scale.  They proved that the
	classical Brown--York mass recovers the ADM mass along such surfaces.
	
	In the present four-dimensional setting, we impose the corresponding
	decay of the trace-free second fundamental form and its first derivative,
	together with the natural scale-invariant controls of the radii,
	diameter, and volume; see
	Definition~\ref{def:BY-AF-nearly-round}.  Thus the hypersurfaces are
	asymptotically round in an intrinsic geometric sense, without being
	assumed to be coordinate spheres or prescribed graphs over them.  The
	next theorem shows that, provided the asymptotic decay of the metric and
	the nearly round rate $\tau$ satisfy a suitable compatibility condition,
	the Brown--York type mass again recovers the ADM mass.
	
	\begin{theorem}[Brown--York type mass along nearly round hypersurfaces]
		\label{thm:BY-AF-nearly-round}
		Let $(M^4,g)$ satisfy the hypotheses of
		Theorem~\ref{thm:BY-AF-large-boundary}, and let $q>1$ be the
		asymptotic decay order in the fixed asymptotically flat chart.
		Let $\{\Sigma_a\}$ be a nearly round family of rate $\tau>0$.
		Assume that, for all sufficiently large $a$, the hypersurface
		$\Sigma_a$ lies in the fixed asymptotically flat chart and encloses
		$B_R$. If moreover
		\begin{align*}
			q+\min\{q,\tau\}>2,
		\end{align*}
		then
		\begin{align}
			\lim_{a\to\infty}m_{\BY}(\Sigma_a)
			=
			m_{\ADM}(g).
		\end{align}
	\end{theorem}
	
	\begin{remark}
		In the special case $\tau=q$, the compatibility condition in
		Theorem~\ref{thm:BY-AF-nearly-round} is automatic. In this sense, at the level of recovering
		the ADM mass along nearly round exhaustions,
		Theorem~\ref{thm:BY-AF-nearly-round} may be regarded as a
		four-dimensional analogue of the result of Shi--Wang--Wu
		\cite{ShiWangWu}.
	\end{remark}

	Theorems~\ref{thm:BY-AF-large-boundary} and~\ref{thm:BY-AF-nearly-round} should be viewed as a first step. They show that the
	Brown--York type mass defined here has the correct large-boundary behavior and
	recovers the ADM mass under the stated shape-defect condition. Further
	questions remain, including positivity and rigidity in analogy with the
	classical Brown--York theory \cite{ShiTam02}, possible monotonicity along
	geometrically natural foliations near infinity \cite{HuiskenIlmanen,HuiskenYau},
	higher-order asymptotic expansions of the mass and Weyl contributions
	\cite{BH,WY}, and the relation with classical embedding-based Brown--York mass
	and Bartnik-type static extensions
	\cite{BrownYork,FanShiTam,LiWeinstein,Bartnik89,MiaoShiTam,Anderson13,McCormick24}.

	\subsection{Idea of proof}
	
	We return to the original
	Hamilton--Jacobi viewpoint of Brown and York \cite{BrownYork}. In that
	formulation, one starts from the gravitational action with the boundary
	three-metric fixed on the timelike boundary. After evaluating the action on a classical solution, one obtains the classical action, but its normalization is not fixed: one may subtract a term $S_0$ depending only on the prescribed boundary data. Thus the reference term should be understood as the choice of the zero level of quasilocal energy from the boundary geometry.
	
	This observation is the guiding point of our construction. In the classical
	Brown--York definition, the subtraction term is computed by isometrically
	embedding the boundary into a flat reference space. In four dimensions such an
	embedding may not exist, but the Hamilton--Jacobi interpretation only requires
	the reference term to be determined by the prescribed boundary metric. We
	therefore recover the flat reference mean curvature intrinsically: for a
	boundary metric with positive sectional curvature, the contracted Gauss
	equation has a unique positive solution, and the trace of this solution defines
	$H_0^+[\sigma]$. Thus the reference term used here depends only on the boundary
	geometry, exactly as suggested by the Brown--York viewpoint. Whenever a flat
	reference embedding exists, this intrinsic reference term agrees with the
	classical one.
	
	We then study the large-boundary behavior of the resulting Brown--York type
	mass. For a family of large hypersurfaces $\{\Sigma_a\}$, we regard the
	Brown--York type mass as a functional of the ambient metric and expand it at
	the Euclidean metric. Since the Euclidean hypersurface has the same physical
	and reference mean curvature, the zeroth-order term vanishes. The first
	variation splits into the usual ADM boundary integral and an additional term
	depending on the Euclidean shape of $\Sigma_a$ and the asymptotic perturbation
	$h=g-\gE$. The quadratic remainder is controlled uniformly by the decay of $h$
	and the scale of $\Sigma_a$.
	
	Therefore the convergence problem is reduced to the shape-dependent correction
	term. The ADM boundary integral converges to $m_{\ADM}(g)$, while the correction
	term measures the failure of the Euclidean geometry of $\Sigma_a$ to behave like
	the round model. For coordinate spheres this correction vanishes identically.
	
	For nearly round hypersurfaces, the intrinsic near-roundness assumptions first
	imply that, in the fixed asymptotically flat chart, the Euclidean geometry of
	$\Sigma_a$ is close to an umbilic constant-curvature model at the natural
	scale. We subtract this model from the correction tensor. The model part
	cancels, and the remaining part is controlled by the deviation from roundness
	together with the decay of $h$. For the terms containing derivatives of the
	Euclidean geometry, integration by parts transfers the derivatives onto the
	decaying perturbation $h$. This gives the required decay of the correction
	term.
	
	Combining these estimates, the shape-dependent correction tends to zero under
	the condition $q+\min\{q,\tau\}>2$. Hence the Brown--York type mass converges
	to the ADM mass along nearly round families.
	
	\subsection{Organization of the paper}
	
	In Section~\ref{sec:BY-mass}, we define the Brown--York type mass used in this paper. We first explain why the classical flat reference embedding may fail in dimension four, and then replace the embedding-based reference term by the intrinsic positive solution of the contracted Gauss equation. In Section~\ref{sec:BY-large-hypersurfaces}, we study the large-boundary limit in a fixed asymptotically flat chart. We derive the expansion of the Brown--York type mass into the ADM boundary integral, a shape-dependent correction term, and a quadratic remainder. The ADM term converges to $m_{\ADM}(g)$, so the remaining issue is to control the correction term. Consequently, we proved Theorem~\ref{thm:BY-AF-large-boundary} in Subsection~\ref{subsec:large-surface-By-mass}. We then show that this correction vanishes for coordinate spheres and decays for nearly round hypersurfaces under a suitable condition, which leads to the proof of Theorem~\ref{thm:BY-AF-nearly-round} in Subsection~\ref{subsec:nearly-round-surface-BY-mass}. Finally, we verify the hypotheses for small radial graphs over large coordinate spheres.

	\subsection{Acknowledgments}
	
	I am thankful for discussions with my supervisor Yin Hao.
	
	\section{Brown--York type mass on four-dimensional AF ends}\label{sec:BY-mass}
	
	In this section we define a Brown--York type mass on four-dimensional
	asymptotically flat ends.  The main issue is that, in contrast with the
	classical three-dimensional setting, the reference term $H_0$ in \eqref{eq:classical-BY} cannot in general be
	defined by isometrically embedding the boundary into a flat background.
	We therefore recover the reference mean curvature intrinsically, by solving
	the contracted Gauss equation for the shape operator.
	
	\subsection{Local isometric embeddability in the Euclidean background}
	\label{subsec:local-embed}
	
	In four dimensions, the existence of a local isometric embedding into the Euclidean background is itself a geometric condition on the boundary metric. To see that this issue genuinely occurs in the asymptotically flat setting,
	we recall a refined asymptotic expansion obtained in \cite{YouALE26}.
	Although that result is formulated for scalar-flat ALE ends, its specialization
	to the trivial group at infinity, $\Gamma=\{1\}$, gives a class of
	scalar-flat AF ends with an AF coordinate system in which
	\begin{align}\label{eq:metric-expansion-after}
		g_{ij}
		=
		\delta_{ij}
		+
		\left(
		\left(\Weyl_\infty\right)_{ik\ell j}
		+
		\Xi(\lambda)_{ijk\ell}
		\right)
		\frac{x_kx_\ell}{|x|^4}
		+
		\mathcal O_\infty(|x|^{-2-\varepsilon}),
		\qquad
		\varepsilon\in(0,1).
	\end{align}
	Here the coefficient of the $|x|^{-2}$ term splits into a scalar part determined by the ADM mass
	and a Weyl curvature term $\Weyl_\infty$ at infinity.  The discussion below uses this refined expansion to exhibit a possible obstruction to flat reference
	embeddability.
	
	Let $\sigma_r$ be the metric induced by $g$ on the coordinate sphere
	$\Sph_r$.  We show that, when $\Weyl_\infty\neq0$, the metric $\sigma_r$ is not
	locally isometrically embeddable into $\R^4$ for all sufficiently large $r$.
	
	Our obstruction calculation uses the local criterion of Li--Weinstein
	\cite{LiWeinstein}.  In the setting relevant here, their Theorem~7 shows
	that, after solving the once-contracted Gauss equation for a candidate shape
	operator, local isometric embeddability into $\R^4$ is equivalent to the
	Codazzi equation for that solution. Thus, once the contracted Gauss equation has been solved, local flat
	embeddability is equivalent to the Codazzi equation for the resulting shape
	operator.
	
	Under the radial identification
	$\Sph^3\ni\omega\mapsto r\omega\in\Sph_r$, set
	$\hat\sigma_r:=r^{-2}\sigma_r$.  The local embeddability problem for the
	large coordinate sphere $(\Sph_r,\sigma_r)$ is then reduced to that for the
	rescaled metric $\hat\sigma_r$, which is close to the round metric on
	$\Sph^3$. From \eqref{eq:metric-expansion-after}, the induced metric has the rescaled expansion
	\begin{align}
		\hat\sigma_r
		=
		\sigma+r^{-2}\tau+\mathcal O(r^{-2-\varepsilon}),
	\end{align}
	where $\sigma$ is the round metric on $\Sph^3$, and
	\begin{align}
		\tau
		=
		\tau^{\mass}+\tau_{\Weyl},
		\quad
		\tau^{\mass}
		=
		-\frac{\lambda}{9}\sigma,
		\quad
		\tau_{\Weyl}(X,Y)
		=
		\Weyl_\infty(X,\omega,\omega,Y),
	\end{align}
	here $X,Y\in T_\omega\Sph^3$ and
	$\lambda=9\,m_{\ADM}(g)$.
	
	Any local isometric embedding of $(\Sph_r,\sigma_r)$ into $\R^4$ would
	have a shape operator $A_r$ satisfying the contracted Gauss equation
	\begin{align*}
		\Ric_{\sigma_r}^{\sharp}
		=
		(\tr A_r)A_r-A_r^2.
	\end{align*}
	Under the radial identification and the rescaling
	$\hat\sigma_r:=r^{-2}\sigma_r$, the corresponding rescaled shape operator
	$\hat A_r:=rA_r$ would therefore satisfy
	\begin{align}\label{eq:BY-prep-rescaled-contracted-Gauss}
		\Ric_{\hat\sigma_r}^{\sharp}
		=
		(\tr \hat A_r)\hat A_r-\hat A_r^2.
	\end{align}
	
	We now consider \eqref{eq:BY-prep-rescaled-contracted-Gauss} independently
	of the existence of an embedding.  Since $\hat\sigma_r$ is sufficiently close
	to the round metric $\sigma$ for large $r$, the Implicit Function Theorem
	gives a unique $\hat\sigma_r$-self-adjoint solution $\hat A_r:=\hat A_0[\hat\sigma_r]\in\End(T\Sph^3)$ near $\Id$ of \eqref{eq:BY-prep-rescaled-contracted-Gauss}.  By the local
	criterion of Li--Weinstein \cite{LiWeinstein}, the metric $\hat\sigma_r$ is
	locally isometrically embeddable in $\R^4$ precisely when this solution
	satisfies the Codazzi equation.  Indeed, if such an embedding exists, its
	rescaled shape operator agrees with $\hat A_r$ by the local uniqueness above.
	
	Define the Codazzi defect as
	\begin{align}
		\mathcal C_r(X,Y,Z)
		:=
		(\nabla_X^{\hat\sigma_r}\hat K_r)(Y,Z)
		-
		(\nabla_Y^{\hat\sigma_r}\hat K_r)(X,Z),
	\end{align}
	where
	$\hat K_r(X,Y):=\hat\sigma_r(\hat A_rX,Y)$. Thus local embeddability is equivalent to $\mathcal C_r=0$, and the leading term of $\mathcal C_r$ identifies the geometric source of any failure of local
	isometric embeddability.
	
	More precisely, solving \eqref{eq:BY-prep-rescaled-contracted-Gauss} near
	$(\sigma,\Id)$ gives a unique $\hat\sigma_r$-self-adjoint endomorphism
	\begin{align*}
		\hat A_r
		=
		\Id+r^{-2}B[\tau]+\mathcal O(r^{-2-\varepsilon}),
	\end{align*}
	where $B[\tau]$ depends linearly on $\tau$. Thus, if we consider the first-order Taylor expansion of the Codazzi
	defect around $(\sigma,\Id)$, a direct consequence is that the first nonzero term in the Codazzi defect
	is governed by the pair $(\tau,B[\tau])$. Therefore, we need to compute $D\mathcal C\big|_{(\sigma,\Id)}$ on $\bigl(\tau^{\mathrm{mass}}, B[\tau^{\mathrm{mass}}]\bigr)$ and $\bigl(\tau_{\Weyl}, B[\tau_{\Weyl}]\bigr)$, respectively. It is natural to determine $B[\tau]$ from the linearization of the once-contracted Gauss equation \eqref{eq:BY-prep-rescaled-contracted-Gauss}. 
	
	First, the mass part makes no contribution to the leading term of the Codazzi defect, i.e. $D\mathcal C\big|_{(\sigma,\Id)}\bigl(\tau^{\mathrm{mass}}, B[\tau^{\mathrm{mass}}]\bigr)$ vanishes. Indeed, $\tau^{\mass}$ represents the infinitesimal change of the induced metric under a radial deformation of the round sphere; hence, it arises from an actual flat hypersurface deformation. Consequently, to determine the first nonzero term of the Codazzi defect of $\hat A_r$, it remains to examine the Weyl part.
	
	Let $b[\tau](X,Y):=\sigma(B[\tau]X,Y)$. Linearizing \eqref{eq:BY-prep-rescaled-contracted-Gauss} at
	$(\sigma,\Id)$ gives
	\begin{align}\label{eq:BY-prep-b-eq}
		b[\tau]+(\tr_{\sigma}b[\tau])\sigma
		=
		(D\Ric_{\sigma})(\tau)-2\tau .
	\end{align}
	The standard description of the Weyl term on
	$\Sph^3$ identifies $\tau_{\Weyl}$ with a sum of traceless left-- and
	right--invariant symmetric $2$-tensors. Using Milnor's Ricci formula for
	left-invariant metrics on $\SU(2)$ (see, e.g., \cite{Milnor76}), one obtains
	\begin{align*}
		(D\Ric_{\sigma})(\tau_{\Weyl})
		=
		6\tau_{\Weyl}.
	\end{align*}
	Substituting this into \eqref{eq:BY-prep-b-eq} gives
	\begin{align}\label{eq:BY-prep-b-Weyl}
		b[\tau_{\Weyl}]
		=
		4\tau_{\Weyl}.
	\end{align}
	Straightforward variational calculations yield that the first nonzero term in the asymptotic expansion of the Codazzi defect of the pair
	$(\hat\sigma_r,\hat A_r)$, denoted by $\mathcal C(\hat\sigma_r,\hat A_r)$, is
	\begin{align}\label{eq:linearization-Codazzi-defect}
		D\mathcal C\big|_{(\sigma,\Id)}
		\bigl(\tau_{\Weyl},B[\tau_{\Weyl}]\bigr)(X,Y,Z)
		=
		(\nabla^{\sigma}_X b[\tau_{\Weyl}])(Y,Z)
		-
		(\nabla^{\sigma}_Y b[\tau_{\Weyl}])(X,Z).
	\end{align}
	On the other hand, differentiating
	$\tau_{\Weyl}(X,Y)=\Weyl_\infty(X,\omega,\omega,Y)$ along the round sphere gives
	\begin{align}\label{eq:BY-prep-Weyl-curl}
		(\nabla_X^{\sigma}\tau_{\Weyl})(Y,Z)
		-
		(\nabla_Y^{\sigma}\tau_{\Weyl})(X,Z)
		=
		-3\Weyl_\infty(X,Y,\omega,Z).
	\end{align}
	Combining \eqref{eq:BY-prep-b-Weyl}, \eqref{eq:linearization-Codazzi-defect} and \eqref{eq:BY-prep-Weyl-curl}, we obtain
	\begin{align}\label{eq:BY-prep-Weyl-Codazzi-defect}
		\mathcal C(\hat\sigma_r,\hat A_r)(X,Y,Z)
		=
		-12r^{-2}\Weyl_\infty(X,Y,\omega,Z)
		+
		o(r^{-2}).
	\end{align}
	Thus, the first obstruction to the Codazzi equation and local isometric embeddability may arise from the Weyl tensor at infinity. The following proposition makes this obstruction explicit.
	
	\begin{proposition}\label{prop:BY-Weyl-obstruction-LW}
		Assume that $\Weyl_\infty\neq 0$. Then for all sufficiently large $r$, the
		boundary metric $\sigma_r$ of large coordinate sphere $\Sph_r$ is not locally isometrically embeddable into the flat Euclidean background.
	\end{proposition}
	
	The assumption $\Weyl_\infty\neq0$ does not by itself imply that the
	particular contraction
	$\Weyl_\infty(X,Y,\omega,Z)$ is nonzero for every choice of
	$\omega\in\Sph^3$ and $X,Y,Z\in T_\omega\Sph^3$. The following elementary observation shows that these contractions nevertheless detect a nonzero Weyl tensor.
	
	\begin{lemma}\label{lem:BY-radial-Weyl-detection}
		Let $R$ be an algebraic curvature tensor on a Euclidean vector space
		$V$ of dimension at least $4$. Suppose that
		\begin{align*}
			R(X,Y,\omega,Z)=0
		\end{align*}
		for every unit vector $\omega\in V$ and every
		$X,Y,Z\in\omega^\perp$. Then $R$ is of constant-curvature type. In
		particular, if $R$ is an algebraic Weyl tensor, then $R=0$.
	\end{lemma}
	
	\begin{proof}
		Write $R_{ijkl}:=R(e_i,e_j,e_k,e_l)$ with respect to an orthonormal
		basis $\{e_1,\ldots,e_n\}$ of $V$. Taking $\omega=e_r$ in the
		hypothesis gives $R_{ijrk}=0$ whenever $i,j,k\neq r$. By the curvature
		symmetries, any component for which one index occurs exactly once can
		be rewritten, up to sign, in this form. Hence $R_{ijkl}=0$ whenever
		one of the indices $i,j,k,l$ occurs exactly once.
		
		For $i\neq j$, let $a_{ij}:=R_{ijij}=a_{ji}$. The preceding observation
		shows that the only possibly nonzero components are those determined by
		the $a_{ij}$'s. Therefore, writing $X=\sum_i x_i e_i,Y=\sum_i y_i e_i,Z=\sum_i z_i e_i$ and $W=\sum_i w_i e_i$, we have
		\begin{align}
			\label{eq:BY-radial-Weyl-detection-diagonal-form}
			R(X,Y,Z,W)
			=
			\sum_{i<j}
			a_{ij}
			(x_i y_j-x_j y_i)
			(z_i w_j-z_j w_i).
		\end{align}
		
		Choose a unit vector $\omega=\sum_i\omega_i e_i$ with $\omega_i\neq0$
		for every $i$. For $X,Y\in\omega^\perp$, define
		$q=\sum_k q_k e_k$, where $q_k:=\sum_{i\neq k}a_{ik}\omega_i(x_i y_k-x_k y_i)$. It follows from
		\eqref{eq:BY-radial-Weyl-detection-diagonal-form} that
		\begin{align*}
			R(X,Y,\omega,Z)=\langle q,Z\rangle
		\end{align*}
		holds for every $Z\in V$. The hypothesis
		implies that $q\perp\omega^\perp$, whereas
		$\langle q,\omega\rangle=R(X,Y,\omega,\omega)=0$. Hence $q=0$, namely,
		\begin{align}\label{eq:qk-identity}
			\sum_{i\neq k}
			a_{ik}\omega_i
			(x_i y_k-x_k y_i)
			=
			0
		\end{align}
		for every $k$ and all $X,Y\in\omega^\perp$.
		
		Fix $k$, and set $u^{(k)}:=\sum_{i\neq k}a_{ik}\omega_i e_i$. Since the
		$k$-th component of $u^{(k)}$ vanishes, the preceding identity \eqref{eq:qk-identity} becomes
		\begin{align*}
			y_k\langle u^{(k)},X\rangle-x_k\langle u^{(k)},Y\rangle=0.
		\end{align*}
		Let
		$E_k:=\omega^\perp\cap e_k^\perp$, and take
		$Y=e_k-\omega_k\omega$. Then $Y\in\omega^\perp$ and
		$y_k=1-\omega_k^2\neq0$. Thus, for $X\in E_k$, we obtain
		$\langle u^{(k)},X\rangle=0$. Consequently,
		$u^{(k)}\in E_k^\perp=\operatorname{span}\{\omega,e_k\}$.
		
		Because the $k$-th component of $u^{(k)}$ is zero, there is a constant
		$\alpha_k$ such that $u^{(k)}=\alpha_k(\omega-\omega_k e_k)$. Comparing
		the $i$-th components for $i\neq k$ gives
		$a_{ik}\omega_i=\alpha_k\omega_i$. Moreover, $\omega_i\neq0$, it follows
		that $a_{ik}=\alpha_k$ for every $i\neq k$. If $i\neq j$, then $a_{ij}=\alpha_j$ and
		$a_{ji}=\alpha_i$, hence all the $\alpha_k$ agree.
		Thus $a_{ij}=c$ for some constant $c$ and every $i\neq j$. Substituting
		this into \eqref{eq:BY-radial-Weyl-detection-diagonal-form} yields
		\begin{align*}
			R(X,Y,Z,W)
			&=
			c\sum_{i<j}
			(x_i y_j-x_j y_i)
			(z_i w_j-z_j w_i) \\
			&=
			c\bigl(
			\langle X,Z\rangle\langle Y,W\rangle
			-
			\langle X,W\rangle\langle Y,Z\rangle
			\bigr).
		\end{align*}
		Hence $R$ is of constant-curvature type.
		
		If $R$ is an algebraic Weyl tensor, then its Ricci contraction vanishes.
		On the other hand, the Ricci contraction of the tensor above is a
		nonzero multiple of $c\,g$. Therefore $c=0$, and hence $R=0$.
	\end{proof}

	\begin{proof}[Proof of Proposition~\ref{prop:BY-Weyl-obstruction-LW}]
		By Lemma~\ref{lem:BY-radial-Weyl-detection}, there exist
		$\omega_0\in\Sph^3$ and
		$X_0,Y_0,Z_0\in T_{\omega_0}\Sph^3$ such that
		\begin{align*}
			\Weyl_\infty(X_0,Y_0,\omega_0,Z_0)\neq0.
		\end{align*}
		Evaluating \eqref{eq:BY-prep-Weyl-Codazzi-defect} at these fixed vectors
		shows that $\mathcal C(\hat\sigma_r,\hat A_r)$ is nonzero for all
		sufficiently large $r$.  Hence the Codazzi equation fails, and the
		Li--Weinstein criterion rules out local isometric embeddability into
		$\R^4$.
	\end{proof}
	
	\subsection{Brown--York type quasilocal energy}
	
	Since isometric embeddings may fail to exist, we next consider the issues that must be addressed when attempting to define a Brown--York type mass in our geometric setting, namely how to choose the reference term once this classical
	reference-hypersurface picture is no longer available.
	
	Recall that, in the Brown--York construction \cite{BrownYork}, the reference term is
	required to depend only on the boundary metric. We now introduce a straightforward construction of a reference term that fulfills this requirement.
	
	Let $\Sigma$ be a smooth closed $3$-manifold. We call
	$(\Sigma,\sigma_\Sigma,A_\Sigma)$ a \emph{base point} if $\sigma_\Sigma$ is a
	Riemannian metric on $\Sigma$, $A_\Sigma$ is a $\sigma_\Sigma$-self-adjoint
	endomorphism of $T\Sigma$, and
	\begin{align}
		\Ric_{\sigma_\Sigma}^{\sharp}
		=
		(\tr A_\Sigma)A_\Sigma
		-
		A_\Sigma^2 .
	\end{align}
	Assume in addition that $A_\Sigma$ is nondegenerate, namely
	$\det A_\Sigma(p)\neq 0$ for every $p\in\Sigma$. Then the contracted Gauss
	equation is stable under small perturbations of the metric. More precisely, for
	every metric $\sigma$ sufficiently close to $\sigma_\Sigma$ in $C^{2,\alpha}$, the Banach-space Implicit Function Theorem yields that there exists a unique endomorphism
	$A_0[\sigma;\sigma_\Sigma,A_\Sigma]$, close to $A_\Sigma$ in $C^{0,\alpha}$,
	such that
	\begin{align}\label{eq:contracted-gauss-general-reference}
		\Ric_{\sigma}^{\sharp}
		=
		(\tr A_0[\sigma;\sigma_\Sigma,A_\Sigma])
		A_0[\sigma;\sigma_\Sigma,A_\Sigma]
		-
		A_0[\sigma;\sigma_\Sigma,A_\Sigma]^2 .
	\end{align}
	The solution depends $C^1$ on $\sigma$ and satisfies
	$A_0[\sigma_\Sigma;\sigma_\Sigma,A_\Sigma]=A_\Sigma$.
	
	After shrinking the neighborhood if necessary, the solution is
	$\sigma$-self-adjoint. If $A$ solves \eqref{eq:contracted-gauss-general-reference},
	then its $\sigma$-adjoint also solves the same equation. By the local uniqueness,
	one has $A=A^{*_\sigma}$.
	
	We also record that the choice of $A_\Sigma$ includes a choice of branch. Indeed,
	the contracted Gauss equation is invariant under $A\mapsto -A$; hence
	$(\Sigma,\sigma_\Sigma,A_\Sigma)$ and
	$(\Sigma,\sigma_\Sigma,-A_\Sigma)$ determine two local branches, with opposite
	reference mean curvatures near $\sigma_\Sigma$. Once the base point is fixed,
	there is no ambiguity within that branch.
	
	In summary, the Implicit Function Theorem enables us to define the reference mean curvature associated with the chosen base point by
	\begin{align}\label{eq:H0-general}
		H_0[\sigma;\sigma_\Sigma,A_\Sigma]
		:=
		\tr A_0[\sigma;\sigma_\Sigma,A_\Sigma].
	\end{align}
	When a flat reference embedding exists with shape operator in the local solution
	class selected by $(\sigma_\Sigma,A_\Sigma)$, \eqref{eq:H0-general} agrees with
	the classical reference mean curvature. The construction, however, remains
	defined without requiring flat embeddability.
	
	Now we may define the Brown--York type mass. Let $(N^4,g)$ be a Riemannian
	$4$-manifold, and let $\Sigma=\partial D\subset N$ be a smooth closed hypersurface
	bounding a relatively compact domain $D$. Suppose that the induced metric
	$\sigma=g|_\Sigma$ is sufficiently close to $\sigma_\Sigma$ in $C^{2,\alpha}$, so
	that \eqref{eq:H0-general} is defined. If $H_g(\Sigma)$ denotes the physics mean curvature of $\Sigma$ in $(N,g)$ with respect to the outward unit normal of $D$, set
	\begin{align}
		m_{\BY}(\Sigma;\sigma_\Sigma,A_\Sigma)
		:=
		\frac{1}{3\omega_3}
		\int_{\Sigma}
		\bigl(
		H_0[\sigma;\sigma_\Sigma,A_\Sigma]
		-
		H_g(\Sigma)
		\bigr)\,d\mu_\sigma,
		\quad
		\omega_3=|\Sph^3|=2\pi^2 .
	\end{align}
	
	\medskip
	
	The construction above is local in the chosen nondegenerate solution
	$(\sigma_\Sigma,A_\Sigma)$. Thus the base point is part of the reference data:
	different choices may give different nearby solutions, and hence different
	reference mean curvatures. In this general form,
	$H_0[\sigma;\sigma_\Sigma,A_\Sigma]$ is therefore not determined by $\sigma$
	alone.
	
	For the large hypersurfaces considered below, there is a canonical choice of
	branch. When the boundary metric has positive sectional curvature, the positive
	definite solution of the contracted Gauss equation is unique and can be written
	explicitly in terms of the intrinsic curvature of the boundary metric.
	
	Specifically, let $(\Sigma^3,\sigma)$ be a Riemannian $3$-manifold, and $\sigma$ has positive sectional curvature at every point. Define the endomorphism
	\begin{align}\label{eq:B-sigma-def}
		B_\sigma
		:=
		\frac12\Scal_\sigma \Id
		-
		\Ric_\sigma^\sharp .
	\end{align}
	Denote the sectional curvature of the $2$-plane spanned by $X,Y$ by $K(X,Y)$. In dimension $3$, if
	$\{e_1,e_2,e_3\}$ is a $\sigma$-orthonormal basis diagonalizing
	$\Ric_\sigma^\sharp$ at a point, then the eigenvalues of $B_\sigma$ at that point are $K(e_2,e_3), K(e_1,e_3)$ and $K(e_1,e_2)$. This implies that $B_\sigma$ is positive definite. In this case we define
	\begin{align}\label{eq:A-plus-def}
		A_0^+[\sigma]
		:=
		\sqrt{\det B_\sigma}\,B_\sigma^{-1},
	\end{align}
	where the determinant is taken for $B_\sigma$ as an endomorphism of $T\Sigma$.
	We also define
	\begin{align}
		H_0^+[\sigma]
		:=
		\tr A_0^+[\sigma].
	\end{align}
	
	\begin{proposition}\label{prop:positive-reference-branch}
		$A_0^+[\sigma]$ is the unique positive definite endomorphism of $T\Sigma$
		satisfying
		\begin{align}\label{eq:positive-branch-contracted-gauss}
			\Ric_\sigma^\sharp
			=
			(\tr A_0^+[\sigma])A_0^+[\sigma]
			-
			(A_0^+[\sigma])^2 .
		\end{align}
		In particular, $H_0^+[\sigma]$ is determined only by $\sigma$ and does not
		depend on a choice of base point.
	\end{proposition}
	
	\begin{proof}
		The statement is pointwise. Fix $p\in\Sigma$ and choose a
		$\sigma$-orthonormal basis in which
		\begin{align*}
			B_\sigma
			=
			\diag(\lambda_1,\lambda_2,\lambda_3),
			\qquad
			\lambda_i>0 .
		\end{align*}
		Set $D:=\det B_\sigma=\lambda_1\lambda_2\lambda_3$. Then
		\begin{align*}
			A_0^+[\sigma]
			=
			\sqrt D\,B_\sigma^{-1}
			=
			\diag\left(
			\frac{\sqrt D}{\lambda_1},
			\frac{\sqrt D}{\lambda_2},
			\frac{\sqrt D}{\lambda_3}
			\right).
		\end{align*}
		Write $a_i:=\tfrac{\sqrt D}{\lambda_i}$. The $i$-th eigenvalue of
		$(\tr_\sigma A_0^+)A_0^+-(A_0^+)^2$ is
		\begin{align*}
			a_i(a_j+a_k)
			=
			\frac{\sqrt D}{\lambda_i}
			\left(
			\frac{\sqrt D}{\lambda_j}
			+
			\frac{\sqrt D}{\lambda_k}
			\right)
			=
			\lambda_j+\lambda_k,
		\end{align*}
		where $\{i,j,k\}=\{1,2,3\}$. On the other hand, the $i$-th eigenvalue of $\Ric_\sigma^\sharp$ is also
		$\lambda_j+\lambda_k$. This proves \eqref{eq:positive-branch-contracted-gauss}.
		
		It remains to establish uniqueness within the class of positive definite solutions. Suppose $A$ is another positive definite solution satisfying the contracted Gauss equation \eqref{eq:positive-branch-contracted-gauss}. Then the endomorphism $\Ric_\sigma^\sharp$ is a polynomial in $A$; consequently, $A$ commutes with $\Ric_\sigma^\sharp$ and with $B_\sigma=\frac12\Scal_\sigma\Id-\Ric_\sigma^\sharp$. Since both operators are $\sigma$-self-adjoint, they are simultaneously diagonalizable. Thus, in a $\sigma$-orthonormal basis, we may write $A=\diag(k_1,k_2,k_3)$, where $k_i>0$. Then the contracted Gauss equation gives
		\begin{align*}
			\Ric_\sigma^\sharp
			=
			\diag\bigl(k_1(k_2+k_3),\,k_2(k_1+k_3),\,k_3(k_1+k_2)\bigr).
		\end{align*}
		Hence
		\begin{align*}
			B_\sigma
			=
			\diag(k_2k_3,k_1k_3,k_1k_2).
		\end{align*}
		It follows that $\det B_\sigma=(k_1k_2k_3)^2$ and
		\begin{align*}
			\sqrt{\det B_\sigma}\,B_\sigma^{-1}
			=
			\diag(k_1,k_2,k_3)=A .
		\end{align*}
		Thus every positive definite solution equals
		$\sqrt{\det B_\sigma}\,B_\sigma^{-1}$, and uniqueness follows.
	\end{proof}
	
	\begin{remark}
		The formula \eqref{eq:A-plus-def} is the intrinsic three-dimensional
		version of the explicit inverse formula for
		$A\mapsto(\tr A)A-A^2$ in Li--Weinstein \cite{LiWeinstein}. In their
		notation, the inverse is written in terms of the eigenvalues of
		$\Ric_\sigma^\sharp$.
	\end{remark}
	
	\medskip
	
	Thus, for a boundary metric with positive sectional curvature, the reference mean
	curvature can be defined directly from the metric itself, without first choosing
	a base point. This leads to the following base-point-free version of the
	Brown--York type mass.
	
	\begin{definition}
		Let $(N^4,g)$ be a Riemannian $4$-manifold, and let
		$\Sigma=\partial D\subset N$ be a smooth closed hypersurface bounding a
		relatively compact domain $D$. Let $\sigma=g|_\Sigma$ be the induced metric,
		and assume that $\sigma$ has positive sectional curvature. We define the
		positive reference mean curvature by
		\begin{align}
			H_0^+[\sigma]
			:=
			\tr
			\left(
			\sqrt{\det B_\sigma}\,B_\sigma^{-1}
			\right),
			\qquad
			B_\sigma
			=
			\frac12\Scal_\sigma\Id-\Ric_\sigma^\sharp.
		\end{align}
		If $H_g(\Sigma)$ denotes the physical mean curvature of $\Sigma$ in $(N,g)$ with
		respect to the outward unit normal of $D$, we define the Brown--York type mass as
		\begin{align}\label{eq:BY-positive-branch}
			m_{\BY}(\Sigma)
			:=
			\frac{1}{3\omega_3}
			\int_\Sigma
			\bigl(
			H_0^+[\sigma]-H_g(\Sigma)
			\bigr)\,d\mu_\sigma,
		\end{align}
		where $\omega_3=|\Sph^3|=2\pi^2$, and $d\mu_\sigma$ denotes the area element.
	\end{definition}
	
	The definition above is also close in spirit to the asymptotically flat
	holographic renormalization of Mann--Marolf \cite{MannMarolf}. In their
	Lorentzian setting, the reference subtraction is chosen as a local and covariant
	counterterm of the boundary geometry, rather than through an auxiliary embedding
	of the boundary into a flat reference spacetime. In the present Riemannian setting, the contracted Gauss equation plays the
	same role: it determines $H_0^+[\sigma]$ from $\sigma$ itself whenever the positive
	branch is defined. Thus the subtraction term in
	\eqref{eq:BY-positive-branch} is intrinsic to the boundary geometry.
	
	\section{The ADM limit problem on large hypersurfaces}
	\label{sec:BY-large-hypersurfaces}
	
	We now turn to the large-boundary limit of the Brown--York type mass on the
	chosen AF end. For a sequence of smooth closed hypersurfaces $\{\Sigma_a\}$ escaping to infinity, we ask when
	\begin{align}
		\lim_{a\to\infty} m_{\BY}(\Sigma_a)
		=
		m_{\ADM}(g).
	\end{align}
	The point is not only to verify convergence along a particular exhaustion, but
	to understand how the limiting value depends on the shape of the large
	hypersurfaces and the geometry of the AF end at infinity.
	
	Having fixed the asymptotically flat coordinate chart, we shall work with the following coordinate-adapted class of large hypersurfaces in $\R^4\setminus B_R$. The assumptions are imposed to make the ADM-limit computation manageable.
	
	\begin{assumption}\label{ass:large-hypersurfaces}
		Let $\{\Sigma_a\}_{a\geq 1}$ be a sequence of smooth, closed, connected hypersurfaces contained in $\R^4\setminus B_R$, each enclosing $B_R$. Equivalently, $\Sigma_a$ and $\partial B_R$ bound a compact region in $\overline{\R^4\setminus B_R}$. Set
		$\rho_a:=\inf_{\Sigma_a}|x|$. We assume the following.
		
		\begin{enumerate}[label=\textup{(\arabic*)}]
			\item \emph{Escape to infinity and radius control.}
			One has $\rho_a\to\infty$ as $a\to\infty$. Moreover, there exists a
			constant $C_{\mathrm{rad}}>0$, independent of $a$, such that
			\begin{align}
				\sup_{\Sigma_a}|x|
				\leq
				C_{\mathrm{rad}}\rho_a .
			\end{align}
			
			\item \emph{Uniform Euclidean convexity at scale $\rho_a$.}
			Let $\kappa_1^E,\kappa_2^E,\kappa_3^E$ be the Euclidean principal
			curvatures of $\Sigma_a$ with respect to the outward Euclidean
			unit normal. There exist constants
			$0<c_{\mathrm{curv}}<C_{\mathrm{curv}}<\infty$, independent of $a$, such that
			\begin{align}\label{eq:large-hypersurface-curvature-control}
				\frac{c_{\mathrm{curv}}}{\rho_a}
				\leq
				\kappa_i^E
				\leq
				\frac{C_{\mathrm{curv}}}{\rho_a},
				\qquad
				i=1,2,3 .
			\end{align}
		\end{enumerate}
	\end{assumption}
	
	Let $\gamma_a:=\gE|_{\Sigma_a}$ and $\sigma_a:=g|_{\Sigma_a}$. The metric $\gamma_a$ is induced
	on this coordinate hypersurface by the Euclidean background $\gE$, whereas
	$\sigma_a$ is the physical metric induced by $g$. Thus, $(\Sigma_a,\sigma_a)$ and $(\Sigma_a,\gamma_a)$ represent the same coordinate hypersurface but equipped with different induced metrics. We denote by
	$K_a^E$ and $A_a^E$ the second fundamental form and the shape operator of the hypersurface $\Sigma_a$ with respect to $\gE$ and the outward Euclidean unit normal. 
	
	By \eqref{eq:large-hypersurface-curvature-control}, the hypersurface
	$\Sigma_a\subset(\R^4\setminus B_R,\gE)$ is strictly convex at scale
	$\rho_a$. Hence the Gauss equation for this Euclidean embedding implies that
	$\gamma_a$ has positive sectional curvature. Therefore, for the metric
	$\gamma_a$, the contracted Gauss equation admits a unique positive definite
	solution $A_0^+[\gamma_a]$ in view of
	Proposition~\ref{prop:positive-reference-branch}. Since the Euclidean shape operator $A_a^E$ is itself positive definite and satisfies this equation, uniqueness gives \begin{align} A_0^+[\gamma_a]=A_a^E, \qquad H_0^+[\gamma_a]=H_{\gE}(\Sigma_a). \end{align}
	
	Moreover, by the asymptotic flatness of $g$, the induced physical metric
	$\sigma_a$ is, after rescaling by $\rho_a^{-1}$, a small
	$C^{2,\alpha}$ perturbation of $\gamma_a$. Hence, for all sufficiently large $a$, the positive solution of
	\eqref{eq:positive-branch-contracted-gauss} remains defined at $\sigma_a$.
	Thus the Brown--York type mass
	\begin{align}\label{eq:BY-general-hypersurface-def}
		m_{\BY}(\Sigma_a)
		:=
		\frac{1}{3\omega_3}
		\int_{\Sigma_a}
		\bigl(
		H_0^+[\sigma_a]-H_g(\Sigma_a)
		\bigr)\,d\mu_{\sigma_a}
	\end{align}
	is well-defined. We shall use \eqref{eq:BY-general-hypersurface-def} to compute
	the limit of the Brown--York type mass for the family of closed hypersurfaces
	$\{\Sigma_a\}_{a\geq 1}$.

	\subsection{Asymptotic expansion of the Brown--York type mass for large hypersurfaces}\label{subsec:large-surface-By-mass}
	
	We derive the asymptotic expansion of the Brown--York type mass by viewing
	the difference between the reference and physical mean curvatures as a
	functional of the ambient metric and linearizing it at the Euclidean
	background. For each $a$, let $\mathcal U_a$ denote the set of smooth
	Riemannian metrics $\bar g$ defined on a neighborhood of $\Sigma_a$ for
	which the induced metric $\bar\sigma_a:=\bar g|_{\Sigma_a}$ has positive sectional curvature. By the uniform Euclidean convexity of
	$\Sigma_a$, the Euclidean metric $\gE$ belongs to $\mathcal U_a$. Moreover,
	by asymptotic flatness and the uniform scaled geometry of $\Sigma_a$, the
	physical metric $g$ also belongs to $\mathcal U_a$ for all sufficiently
	large $a$.
	
	For
	$\bar g\in\mathcal U_a$, define
	\begin{align}
		\mathcal Q_a(\bar g)
		:=
		\int_{\Sigma_a}
		\bigl(
		H_0^+[\bar\sigma_a]
		-
		H_{\bar g}(\Sigma_a)
		\bigr)
		\,d\mu_{\bar\sigma_a}.
	\end{align}
	Then, whenever $g\in\mathcal U_a$, one has
	\begin{align*}
		m_{\BY}(\Sigma_a)
		=
		\frac{1}{3\omega_3}\mathcal Q_a(g).
	\end{align*}
	Actually, for all sufficiently large $a$, both $g$ and
	the path $g_t:=\gE+t(g-\gE)$, where $0\leq t\leq 1$, belong to $\mathcal U_a$. Therefore $\mathcal Q_a(g)$ can be computed by a
	Taylor expansion of $\mathcal Q_a$ at $\gE$ along this path. Furthermore, since a Euclidean hypersurface has the same physical and reference mean
	curvature, we have $\mathcal Q_a(\gE)=0$. Thus the leading contribution to $m_{\BY}(\Sigma_a)$ comes from the first
	variation of $\mathcal Q_a$ at $\gE$.
	
	Before deriving the variation formula, we introduce the notation used below.
	Let $(\Sigma,\gamma)$ be a closed $3$-manifold with positive sectional
	curvature. Recall that
	\begin{align*}
		B_\gamma
		:=
		\frac12\Scal_\gamma\Id-\Ric_\gamma^\sharp,
		\quad
		A_0^+[\gamma]
		:=
		\sqrt{\det B_\gamma}\,B_\gamma^{-1},
		\quad H_0^+[\gamma]
		:=
		\tr A_0^+[\gamma].
	\end{align*}
	At the Euclidean induced metric $\gamma_a$, we abbreviate
	\begin{align*}
		A_0:=A_0^+[\gamma_a],
		\qquad
		H_0:=H_0^+[\gamma_a],
		\qquad
		B:=B_{\gamma_a},
	\end{align*}
	and define
	\begin{align*}
		P
		:=
		\frac12H_0B^{-1}-B^{-1}A_0,
		\qquad
		p:=\tr P.
	\end{align*}
	Since $B$ and $A_0$ are commuting self-adjoint endomorphisms, $P$ is
	self-adjoint. Using $\gamma_a$, we identify $P$ with a symmetric
	$(2,0)$-tensor, whose components are denoted by $P^{\alpha\beta}$.
	
	In what follows, Greek indices
	$\alpha,\beta,\gamma,\delta\in\{1,2,3\}$ denote tangential indices on
	$\Sigma_a$, and all covariant derivatives, traces, Laplacians, and curvature
	tensors are taken with respect to $\gamma_a$. For a symmetric $(2,0)$-tensor
	$P$, define
	\begin{align*}
		\mathcal L_{\gamma_a}(P)^{\alpha\beta}
		:={}&
		\frac12\nabla^\alpha\nabla^\beta p
		-
		\frac12\gamma_a^{\alpha\beta}\Delta p
		+
		\frac12\Delta P^{\alpha\beta}
		+
		\frac12\gamma_a^{\alpha\beta}
		\nabla_\delta\nabla_\gamma P^{\gamma\delta}
		\\
		&-
		\frac12\nabla^\alpha\nabla_\delta P^{\delta\beta}
		-
		\frac12\nabla^\beta\nabla_\delta P^{\alpha\delta},
	\end{align*}
	and
	\begin{align}\label{eq:BY-def-EP}
		\mathcal E_{\gamma_a}(P)^{\alpha\beta}
		:=
		\mathcal L_{\gamma_a}(P)^{\alpha\beta}
		-
		\frac12p\,\Ric_{\gamma_a}^{\alpha\beta}
		+
		P^{\gamma\delta}
		R(\gamma_a)_\gamma{}^\alpha{}_\delta{}^\beta.
	\end{align}
	Finally, set
	\begin{align}\label{eq:BY-def-D}
		\mathfrak D_a^{\alpha\beta}
		:=
		\mathcal E_{\gamma_a}(P)^{\alpha\beta}
		+
		\frac12A_0^{\alpha\beta}.
	\end{align}
	Here $A_0^{\alpha\beta}$ means the $(2,0)$-tensor obtained from the second fundamental form corresponding to the endomorphism $A_0$ by raising both indices with $\gamma_a$. One may verify directly that $\mathfrak D_a$ is symmetric.
	
	We are now ready to state and prove the variation formula.
	
	\begin{lemma}\label{lem:BY-first-variation-general-hypersurface}
		Let $k$ be a smooth symmetric $2$-tensor defined on a neighborhood of $\Sigma_a$ in $\R^4\setminus B_R$, and set $\bar g_t:=\gE+tk$. Then
		\begin{align}\label{eq:BY-first-variation-with-defect-Q}
			(D\mathcal Q_a)_{\gE}(k)
			=
			\frac12
			\int_{\Sigma_a}
			\bigl(
			\partial_j k_{ij}
			-
			\partial_i k_{jj}
			\bigr)(\nu_E)^i\,d\mu_E
			+
			\int_{\Sigma_a}
			\mathfrak D_a^{\alpha\beta}k_{\alpha\beta}\,d\mu_E .
		\end{align}
		Here $k_{ab}$ denotes the tangential restriction of $k$ to $\Sigma_a$, and
		$\nu_E,d\mu_E$ are computed with respect to the Euclidean background.
	\end{lemma}
	
	\begin{proof}
		Set $\gamma_{a,t}:=\bar g_t|_{\Sigma_a}$ and $q:=k|_{\Sigma_a}$, then $\gamma_{a,t}
		=
		\gamma_a+tq$. All indices in the following computation are raised and lowered using
		$\gamma_a$.
		
		Differentiating $H_0^+[\gamma]
		=
		\tr
		\left(
		\sqrt{\det B_\gamma}B_\gamma^{-1}
		\right)$ at $\gamma_a$ gives
		\begin{align*}
			\dot H_0
			=
			\frac12p\,\dot{\Scal}
			+
			P^\beta{}_\alpha
			q^{\alpha\gamma}
			\Ric_{\gamma\beta}
			-
			P^{\gamma\beta}
			\dot\Ric_{\gamma\beta}.
		\end{align*}
		Substituting the standard first variation formula for scalar and Ricci curvature, and integrating by parts twice on the closed hypersurface $\Sigma_a$, yields
		\begin{align}\label{eq:BY-reference-variation-with-E}
			\left.
			\frac{d}{dt}
			\right|_{t=0}
			\int_{\Sigma_a}
			H_0^+[\gamma_{a,t}]\,d\mu_{\gamma_{a,t}}
			=
			\int_{\Sigma_a}
			\left(
			\mathcal E_{\gamma_a}(P)^{\alpha\beta}
			+
			\frac12H_0\gamma_a^{\alpha\beta}
			\right)
			k_{\alpha\beta}\,d\mu_E .
		\end{align}
		
		We next consider the physical mean-curvature term. Using the standard first variation formula for the Einstein--Hilbert functional with the Gibbons--Hawking--York boundary
		term, as in \cite[(2.12)--(2.13)]{Anderson13}, one obtains
		\begin{align}\label{eq:BY-physical-H-variation-correct}
			\begin{split}
				\left.
				\frac{d}{dt}
				\right|_{t=0}
				\int_{\Sigma_a}
				H_{\bar g_t}(\Sigma_a)\,d\mu_{\gamma_{a,t}}
				&=
				\frac12
				\int_{\Sigma_a}
				\left(
				H_0\gamma_a^{\alpha\beta}
				-
				A_0^{\alpha\beta}
				\right)
				k_{\alpha\beta}\,d\mu_E
				\\
				&\quad
				-
				\frac12
				\int_{\Sigma_a}
				\bigl(
				\partial_jk_{ij}
				-
				\partial_i k_{jj}
				\bigr)
				(\nu_E)^i\,d\mu_E .
			\end{split}
		\end{align}
		Since the background metric is Euclidean, the linearization of the scalar curvature is a divergence; integrating it over the region enclosed by $\Sigma_a$ and applying the divergence theorem yields the displayed boundary integral. Subtracting
		\eqref{eq:BY-physical-H-variation-correct} from
		\eqref{eq:BY-reference-variation-with-E} gives \eqref{eq:BY-first-variation-with-defect-Q}.
	\end{proof}

	\begin{remark}
		If the first variation of the Brown--York type functional $\mathcal Q_a$ were given only by
		the ADM boundary integral (i.e., the first term on the right-hand side of \eqref{eq:BY-first-variation-with-defect-Q}), then $\mathfrak D_a$ would vanish. This cancellation holds for round coordinate
		spheres, but it does not hold for a general convex hypersurface. Thus
		$\mathfrak D_a$ detects the extra
		contribution coming from the non-round geometry of the hypersurface.
	\end{remark}
	
	\medskip
	
	The tensor $\mathfrak D_a$ has a useful interpretation: it measures the part of
	the first variation of $\mathcal Q_a$ at the Euclidean metric which is not
	explained by the Euclidean embedding of $\Sigma_a$. To make this precise, recall
	the linearized metric change produced by moving the Euclidean hypersurface
	$\Sigma_a\subset\R^4$. If the variation vector field is written as
	$Z=Y+f\nu_E$, where $Y$ is tangent to $\Sigma_a$ and $f$ is a function, then the
	induced Euclidean metric changes by
	\begin{align}\label{eq:BY-linearized-embedding-map}
		\mathscr B_a(Y,f)
		=
		\Lie_Y\gamma_a+2fK_a^E .
	\end{align}
	Here $K_a^E$ is the Euclidean second fundamental form.
	
	Since a Euclidean hypersurface has the same reference and physical mean
	curvature, the Brown--York difference vanishes identically along such Euclidean
	deformations. Moreover, the ADM boundary term in
	\eqref{eq:BY-first-variation-with-defect-Q} also vanishes for a pure Euclidean
	diffeomorphism variation $k=\Lie_Z\gE$, because
	$(DR)_{\gE}(\Lie_Z\gE)=0$. Therefore the remaining defect term must be
	orthogonal to all variations of the form
	\eqref{eq:BY-linearized-embedding-map}. This gives the following cokernel
	condition for $\mathfrak D_a$.
	
	\begin{lemma}\label{lem:BY-defect-cokernel}
		The tensor $\mathfrak D_a$ satisfies
		\begin{align}\label{eq:BY-defect-cokernel-equations}
			\nabla_a\mathfrak D_a^{ab}=0,
			\qquad
			K^E_{a,ab}\mathfrak D_a^{ab}=0 .
		\end{align}
		Equivalently,
		\begin{align}\label{eq:integral-orthogonality}
			\int_{\Sigma_a}
			\mathfrak D_a^{ab}
			\bigl(\Lie_Y\gamma_a+2fK_a^E\bigr)_{ab}\,d\mu_E
			=
			0
		\end{align}
		for every tangent vector field $Y$ and every function $f$ on $\Sigma_a$.
	\end{lemma}
	
	\begin{proof}
		We first prove the integral orthogonality \eqref{eq:integral-orthogonality}. Along $\Sigma_a$, write
		$Z=Y+f\nu_E$ as above, and extend $Z$ smoothly to a neighbourhood of $\Sigma_a$ in the
		Euclidean end. Let $\Phi_t$ be the local flow of $Z$ and set
		$k=\Lie_Z\gE$. Then
		\begin{align*}
			\Phi_t^*\gE=\gE+t\Lie_Z\gE+O(t^2),
		\end{align*}
		so this variation is
		generated by an ambient Euclidean diffeomorphism.
		
		For this special variation, the Brown--York type difference has zero first
		variation. Indeed, $\Sigma_a$ with the metric $(\Phi_t^*\gE)|_{\Sigma_a}$ is
		isometric, through $\Phi_t$, to the Euclidean hypersurface
		$\Phi_t(\Sigma_a)\subset(\mathbb R^4,\gE)$. Hence the physical mean curvature
		of $\Sigma_a$ in $(\mathbb R^4,\Phi_t^*\gE)$ is the pull-back of the Euclidean
		mean curvature of $\Phi_t(\Sigma_a)$. The Euclidean shape operator of
		$\Phi_t(\Sigma_a)$ solves the contracted Gauss equation and, for small $t$, is
		the chosen local solution defining $H_0^+$ near
		$(\Sigma_a,\gamma_a,K_a^E)$. Therefore
		\begin{align}
			H_0^+\bigl[(\Phi_t^*\gE)|_{\Sigma_a}\bigr]
			=
			H_{\Phi_t^*\gE}(\Sigma_a)
		\end{align}
		for all sufficiently small $t$. Consequently,
		\begin{align}\label{eq:BY-defect-proof-zero-first-variation}
			\left.
			\frac{d}{dt}
			\right|_{t=0}
			\int_{\Sigma_a}
			\Bigl(
			H_0^+\bigl[(\Phi_t^*\gE)|_{\Sigma_a}\bigr]
			-
			H_{\Phi_t^*\gE}(\Sigma_a)
			\Bigr)
			\,d\mu_{(\Phi_t^*\gE)|_{\Sigma_a}}
			=
			0.
		\end{align}
		
		We also need the ADM boundary integral, namely, the first term on the right-hand side of \eqref{eq:BY-first-variation-with-defect-Q}, to vanish for
		$k=\Lie_Z\gE$. Since scalar curvature is natural under diffeomorphisms, then $(DR)_{\gE}(k)$ vanishes. On the other
		hand, in Euclidean coordinates,
		$(DR)_{\gE}(k)$ is given by $\partial_i(\partial_jk_{ij}-\partial_i k_{jj})$. Applying the
		divergence theorem to the region enclosed by $\Sigma_a$ gives
		\begin{align}\label{eq:BY-defect-proof-zero-ADM-boundary}
			\int_{\Sigma_a}
			\bigl(
			\partial_j k_{ij}
			-
			\partial_i k_{jj}
			\bigr)(\nu_E)^i\,d\mu_E
			=
			0 .
		\end{align}
		If there are other boundary components in the chosen exterior region, we choose
		the extension of $Z$ to be supported in a collar of $\Sigma_a$, so that no
		extra boundary contribution appears.
		
		Now apply \eqref{eq:BY-first-variation-with-defect-Q} to
		$k=\Lie_Z\gE$. Combining
		\eqref{eq:BY-defect-proof-zero-first-variation} with
		\eqref{eq:BY-defect-proof-zero-ADM-boundary}, we obtain
		\begin{align*}
			\int_{\Sigma_a}\mathfrak D_a^{ab}k_{ab}\,d\mu_E=0 .
		\end{align*}
		It remains only to identify $k_{ab}$. For tangent vector fields $X_1,X_2$ on
		$\Sigma_a$, using $Z=Y+f\nu_E$, one has
		\begin{align*}
			(\Lie_Z\gE)(X_1,X_2)
			&=
			(\Lie_Y\gamma_a)(X_1,X_2)
			+
			2fK_a^E(X_1,X_2).
		\end{align*}
		Thus $k_{ab}$ can be expressed as $(\Lie_Y\gamma_a+2fK_a^E)_{ab}$, and hence
		\begin{align}\label{eq:BY-defect-proof-integral-orthogonality}
			\int_{\Sigma_a}
			\mathfrak D_a^{ab}
			\bigl(
			\Lie_Y\gamma_a+2fK_a^E
			\bigr)_{ab}
			\,d\mu_E
			=
			0
		\end{align}
		for every $Y$ and $f$. This proves \eqref{eq:integral-orthogonality}. Next, the two pointwise equations in \eqref{eq:BY-defect-cokernel-equations} are also direct consequences, which can be derived by setting $f$ and $Y$ in \eqref{eq:BY-defect-proof-integral-orthogonality} to vanish, respectively.
	\end{proof}
	
	\medskip
	
	We now return to the main task of this section, namely to determine when
	the Brown--York type mass of the family $\{\Sigma_a\}$ in
	Assumption~\ref{ass:large-hypersurfaces} converges to the ADM mass. We need to derive the asymptotic expansion for its Brown--York type mass. Lemma~\ref{lem:BY-first-variation-general-hypersurface} gives
	the linear contribution at the Euclidean metric, while
	Lemma~\ref{lem:BY-defect-cokernel} identifies the orthogonality
	relations satisfied by the tensor $\mathfrak D_a$ which will later be
	used to estimate the resulting tangential term.
	
	Before applying the first-variation formula \eqref{eq:BY-first-variation-with-defect-Q}, we first justify a
	uniform Taylor expansion of $\mathcal Q_a$ along the affine path from
	$\gE$ to $g$. Set
	\begin{align*}
		h:=g-\gE,
		\qquad
		g_t:=\gE+t h,
		\qquad
		t\in[0,1].
	\end{align*}
	The following proposition shows that, for all sufficiently large $a$,
	the whole path ${g_t},t\in[0,1]$ remains in $\mathcal U_a$ (recall $\mathcal U_a$ is the set of smooth
	Riemannian metrics $\bar g$ defined on a neighborhood of $\Sigma_a$ for
	which the induced metric $\bar\sigma_a:=\bar g|_{\Sigma_a}$ has positive sectional curvature) and
	that the quadratic Taylor remainder is uniformly negligible.
	
	\begin{proposition}\label{prop:BY-uniform-quadratic-remainder}
		Let $h:=g-\gE$ and $g_t:=\gE+t h$, $t\in[0,1]$. For a family of smooth closed hypersurfaces $\{\Sigma_a\}$ satisfying Assumption~\ref{ass:large-hypersurfaces}, one has
		$g_t\in\mathcal U_a$ for all $t\in [0,1]$ and all sufficiently large $a$. Moreover, there exists a constant $C>0$, independent of $a$ and
		$t$, such that
		\begin{align}\label{eq:BY-uniform-quadratic-remainder}
			\sup_{t\in[0,1]}
			\left|
			\frac{d^2}{dt^2}\mathcal Q_a(g_t)
			\right|
			\leq
			C\rho_a^{2-2q}.
		\end{align}
		Consequently,
		\begin{align}\label{eq:BY-Qa-Taylor-without-regularity}
			\mathcal Q_a(g)
			=
			(D\mathcal Q_a)_{\gE}(h)
			+
			\mathcal O(\rho_a^{2-2q}).
		\end{align}
	\end{proposition}
	
	\begin{proof}
		Set $\gamma_a:=\gE|_{\Sigma_a}$ and $\sigma_{a,t}:=g_t|_{\Sigma_a}$. Since $|x|\geq\rho_a$ on $\Sigma_a$, asymptotic flatness gives
		\begin{align}\label{eq:BY-Taylor-AF-estimates}
			|\partial^\ell h|_{\gE}
			\leq
			C\rho_a^{-q-\ell},
			\qquad
			\ell=0,1,2.
		\end{align}
		In particular, on a neighborhood of $\Sigma_a$ contained in
		$\{|x|\geq \rho_a/2\}$, one has $|h|_{\gE}\leq C\rho_a^{-q}\ll 1$ for all
		sufficiently large $a$. Hence the metrics $g_t$ are Riemannian and uniformly
		equivalent to $\gE$ near $\Sigma_a$, uniformly for $t\in[0,1]$.
		
		Let $K_{a,t}$, $A_{a,t}$, $H_{a,t}$, and $\nu_{a,t}$ denote,
		respectively, the second fundamental form, shape operator, mean curvature,
		and outward unit normal of $\Sigma_a$ with respect to $g_t$. Condition
		\textup{(2)} of Assumption~\ref{ass:large-hypersurfaces} gives
		\begin{align*}
			|K_a^E|_{\gamma_a}
			\leq
			C\rho_a^{-1}.
		\end{align*}
		The standard comparison formulas for unit normals and Levi--Civita
		connections, applied to the fixed hypersurface $\Sigma_a$ and the path
		$g_t=\gE+t h$, give
		\begin{align}\label{eq:BY-Taylor-K-variation}
			|K_{a,t}|_{\gamma_a}
			\leq
			C\rho_a^{-1},
			\quad
			|\partial_tK_{a,t}|_{\gamma_a}
			\leq
			C\rho_a^{-1-q},
			\quad
			|\partial_t^2K_{a,t}|_{\gamma_a}
			\leq
			C\rho_a^{-1-2q}.
		\end{align}
		Since $H_{a,t}=\operatorname{tr}_{\sigma_{a,t}}K_{a,t}$,
		\eqref{eq:BY-Taylor-AF-estimates} and
		\eqref{eq:BY-Taylor-K-variation} imply
		\begin{align}\label{eq:BY-Taylor-physical-H-variation}
			|H_{a,t}|
			\leq
			C\rho_a^{-1},
			\quad
			|\partial_tH_{a,t}|
			\leq
			C\rho_a^{-1-q},
			\quad
			|\partial_t^2H_{a,t}|
			\leq
			C\rho_a^{-1-2q}.
		\end{align}
		Similarly, the ambient curvature of $g_t$ satisfies
		\begin{align}\label{eq:BY-Taylor-ambient-curvature-variation}
			|\operatorname{Rm}(g_t)|_{\gE}
			\leq
			C\rho_a^{-2-q},
			\quad
			|\partial_t\operatorname{Rm}(g_t)|_{\gE}
			\leq
			C\rho_a^{-2-q},
			\quad
			|\partial_t^2\operatorname{Rm}(g_t)|_{\gE}
			\leq
			C\rho_a^{-2-2q}.
		\end{align}
		For the corresponding endomorphism $B_{a,t}$ defined as \eqref{eq:B-sigma-def}, by the Gauss equation, $B_{a,t}$ is a universal algebraic
		expression in $\sigma_{a,t}$, $\sigma_{a,t}^{-1}$, $K_{a,t}$, and the tangential restriction of $\Rm(g_t)$. Hence
		\eqref{eq:BY-Taylor-AF-estimates},
		\eqref{eq:BY-Taylor-K-variation}, and
		\eqref{eq:BY-Taylor-ambient-curvature-variation} imply
		\begin{align}\label{eq:BY-Taylor-Bendomorphism-variation}
			|B_{a,t}-B_{a,0}|_{\gamma_a}
			\leq
			C\rho_a^{-2-q},
			\quad
			|\partial_tB_{a,t}|_{\gamma_a}
			\leq
			C\rho_a^{-2-q},
			\quad
			|\partial_t^2B_{a,t}|_{\gamma_a}
			\leq
			C\rho_a^{-2-2q}.
		\end{align}
		
		At $t=0$, choose a $\gamma_a$-orthonormal frame diagonalizing $A_a^E$.
		Then
		\begin{align*}
			B_{a,0}
			=
			\operatorname{diag}
			\bigl(
			\kappa_2^E\kappa_3^E,
			\kappa_1^E\kappa_3^E,
			\kappa_1^E\kappa_2^E
			\bigr).
		\end{align*}
		Combining condition \textup{(2)} of Assumption~\ref{ass:large-hypersurfaces} with \eqref{eq:BY-Taylor-Bendomorphism-variation}, we obtain, for all sufficiently large $a$,
		\begin{align}\label{eq:BY-Taylor-Btensor-positive}
			c\rho_a^{-2}\gamma_a
			\leq
			B_{a,t}
			\leq
			C\rho_a^{-2}\gamma_a,
			\qquad
			t\in[0,1].
		\end{align}
		Since $\sigma_{a,t}$ and $\gamma_a$ are uniformly equivalent, this is
		equivalent to
		\begin{align*}
			c\rho_a^{-2}\sigma_{a,t}
			\leq
			B_{a,t}
			\leq
			C\rho_a^{-2}\sigma_{a,t}.
		\end{align*}
		Thus $B_{a,t}$ is positive definite as a $\sigma_{a,t}$-self-adjoint
		endomorphism. In dimension three, positivity of $B_{a,t}$ is equivalent to
		positivity of the sectional curvatures of $\sigma_{a,t}$. Hence
		$g_t\in\mathcal U_a$ for all $t\in[0,1]$ and all sufficiently large $a$.
		
		Set $\widetilde B_{a,t}:=\rho_a^2B_{a,t}$. By \eqref{eq:BY-Taylor-Btensor-positive}, the endomorphisms
		$\widetilde B_{a,t}$ remain in a fixed compact subset of the positive cone.
		Moreover,
		\eqref{eq:BY-Taylor-Bendomorphism-variation} gives
		\begin{align*}
			|\partial_t\widetilde B_{a,t}|_{\gamma_a}
			\leq
			C\rho_a^{-q},
			\qquad
			|\partial_t^2\widetilde B_{a,t}|_{\gamma_a}
			\leq
			C\rho_a^{-2q}.
		\end{align*}
		Since
		\begin{align*}
			H_0^+[\sigma_{a,t}]
			&=
			\operatorname{tr}
			\left(
			\sqrt{\det B_{a,t}}\,
			B_{a,t}^{-1}
			\right)
			\\
			&=
			\rho_a^{-1}
			\operatorname{tr}
			\left(
			\sqrt{\det\widetilde B_{a,t}}\,
			\widetilde B_{a,t}^{-1}
			\right),
		\end{align*}
		the smoothness of the map
		$B\mapsto\operatorname{tr}(\sqrt{\det B}\,B^{-1})$ on the positive cone
		implies
		\begin{align}\label{eq:BY-Taylor-reference-H-variation}
			|H_0^+[\sigma_{a,t}]|
			\leq
			C\rho_a^{-1},
			\quad
			|\partial_tH_0^+[\sigma_{a,t}]|
			\leq
			C\rho_a^{-1-q},
			\quad
			|\partial_t^2H_0^+[\sigma_{a,t}]|
			\leq
			C\rho_a^{-1-2q}.
		\end{align}

		Write $d\mu_{\sigma_{a,t}}=
		J_{a,t}\,d\mu_{\gamma_a}$. Since $\sigma_{a,t}=\gamma_a+t h|_{T\Sigma_a}$, we have
		\begin{align}\label{eq:BY-Taylor-area-variation}
			|J_{a,t}|
			\leq
			C,
			\quad
			|\partial_tJ_{a,t}|
			\leq
			C\rho_a^{-q},
			\quad
			|\partial_t^2J_{a,t}|
			\leq
			C\rho_a^{-2q}.
		\end{align}
		It follows from
		\eqref{eq:BY-Taylor-reference-H-variation},
		\eqref{eq:BY-Taylor-physical-H-variation}, and
		\eqref{eq:BY-Taylor-area-variation} that
		\begin{align}\label{eq:BY-Taylor-integrand-variation}
			\left|
			\frac{d^2}{dt^2}
			\left[
			\bigl(H_0^+[\sigma_{a,t}]-H_{a,t}\bigr)J_{a,t}
			\right]
			\right|
			\leq
			C\rho_a^{-1-2q}.
		\end{align}
		
		It remains to bound the Euclidean area of $\Sigma_a$. Since all Euclidean
		principal curvatures are positive, the Hadamard theorem implies that the
		Euclidean Gauss map of $\Sigma_a$ is a diffeomorphism onto $\mathbb S^3$.
		Hence
		\begin{align*}
			\omega_3
			=\int_{\Sigma_a}
			\det A_a^E\,d\mu_{\gamma_a}
			=\int_{\Sigma_a}
			\kappa_1^E\kappa_2^E\kappa_3^E\,d\mu_{\gamma_a}
			\geq c\rho_a^{-3} \Vol_{\gamma_a}(\Sigma_a),
		\end{align*}
		which gives
		\begin{align}\label{eq:BY-Taylor-area-bound}
			\Vol_{\gamma_a}(\Sigma_a)
			\leq
			C\rho_a^3.
		\end{align}
		
		Integrating \eqref{eq:BY-Taylor-integrand-variation} and applying \eqref{eq:BY-Taylor-area-bound} yields \eqref{eq:BY-uniform-quadratic-remainder}.
		
		Finally, $\mathcal Q_a(\gE)=0$, since the Euclidean hypersurface has the
		same physical and reference mean curvature. Taylor's formula gives
		\begin{align*}
			\mathcal Q_a(g)
			=
			(D\mathcal Q_a)_{\gE}(h)
			+
			\int_0^1
			(1-t)
			\frac{d^2}{dt^2}\mathcal Q_a(g_t)\,dt.
		\end{align*}
		Equation \eqref{eq:BY-Qa-Taylor-without-regularity} follows from
		\eqref{eq:BY-uniform-quadratic-remainder}.
	\end{proof}
	
	In particular, the Taylor remainder is $o(1)$ as $a\to\infty$. Applying Lemma~\ref{lem:BY-first-variation-general-hypersurface} and Proposition~\ref{prop:BY-uniform-quadratic-remainder}, we obtain
	\begin{align}\label{eq:BY-large-Taylor}
		\begin{split}
			3\omega_3\,m_{\BY}(\Sigma_a)
			=
			\mathcal Q_a(g)
			&=
			\frac12
			\int_{\Sigma_a}
			\bigl(
			\partial_jh_{ij}
			-
			\partial_ih_{jj}
			\bigr)(\nu_E)^i\,d\mu_E
			\\
			&\quad
			+
			\int_{\Sigma_a}
			\mathfrak D_a^{\alpha\beta}
			h_{\alpha\beta}\,d\mu_E
			+
			\mathcal O(\rho_a^{2-2q}).
		\end{split}
	\end{align}
	Here $h_{\alpha\beta}$ denotes the tangential restriction of $h$ to
	$\Sigma_a$.
	
	We now proceed to evaluate the limit of each term in the expansion \eqref{eq:BY-large-Taylor} as $a \to \infty$. For the first term in this expansion, we have the following result:
	
	\begin{lemma}\label{lem:AF-flux-limit}
		In fact,
		\begin{align}\label{eq:AF-flux-limit}
			\lim_{a\to\infty}
			\frac12
			\int_{\Sigma_a}
			\bigl(
			\partial_jh_{ij}
			-
			\partial_ih_{jj}
			\bigr)
			(\nu_E)^i\,d\mu_E
			=
			3\omega_3m_{\ADM}(g).
		\end{align}
	\end{lemma}
	
	\begin{proof}
		Set $V_i:=\partial_jh_{ij}-\partial_ih_{jj}$. Let $\Omega_a^E\subset\R^4$ be the bounded Euclidean region enclosed by
		$\Sigma_a$. Obviously, one has $B_{\rho_a}\subset\Omega_a^E$.
		
		Set $r_a:=\rho_a/2$. For all sufficiently large $a$, the Euclidean
		divergence theorem applied to $\Omega_a^E\setminus B_{r_a}$ gives
		\begin{align}\label{eq:AF-flux-comparison}
			\int_{\Sigma_a}
			V_i(\nu_E)^i\,d\mu_E
			-
			\int_{\Sph_{r_a}}
			V_i(\nu_E)^i\,d\mu_E
			=
			\int_{\Omega_a^E\setminus B_{r_a}}
			\partial_iV_i\,dx .
		\end{align}
		The scalar-curvature expansion in the asymptotically flat coordinates
		gives
		\begin{align*}
			\partial_iV_i
			=
			R_g
			+
			\mathcal O\bigl(
			|h||\partial^2h|
			+
			|\partial h|^2
			\bigr).
		\end{align*}
		Since $g-\gE\in C^{2,\alpha}_{-q}$ with $q>1$, the error term is of order $\mathcal O(|x|^{-2q-2})$, which is integrable on $\R^4\setminus B_R$. Moreover,
		$R_g\in L^1(M,dV_g)$, and asymptotic flatness implies that $dV_g$ and $dx$ are uniformly equivalent near infinity, hence $\partial_iV_i \in L^1(\R^4\setminus B_R,dx)$. It follows from \eqref{eq:AF-flux-comparison} that
		\begin{align*}
			\left|
			\int_{\Sigma_a}
			V_i(\nu_E)^i\,d\mu_E
			-
			\int_{\Sph_{r_a}}
			V_i(\nu_E)^i\,d\mu_E
			\right|
			\leq
			\int_{\R^4\setminus B_{r_a}}
			|\partial_iV_i|\,dx
			\longrightarrow
			0,
		\end{align*}
		because $r_a\to\infty$. On the other hand, by the definition of the
		four-dimensional ADM mass,
		\begin{align*}
			\lim_{a\to\infty}
			\int_{\Sph_{r_a}}
			V_i(\nu_E)^i\,d\mu_E
			=
			6\omega_3m_{\ADM}(g).
		\end{align*}
		Combining the last two relations proves
		\eqref{eq:AF-flux-limit}.
	\end{proof}
	
	With Lemma~\ref{lem:AF-flux-limit} in hand, in view of expansion \eqref{eq:BY-large-Taylor} for the Brown--York type mass of the large hypersurface family $\{\Sigma_a\}$, the convergence of $m_{\BY}(\Sigma_a)$ to the ADM mass is reduced to
	showing that the correction term $\int_{\Sigma_a}
	\mathfrak D_a^{\alpha\beta}
	h_{\alpha\beta}d\mu_E$ on the right-hand side of \eqref{eq:BY-large-Taylor} tends to zero.
	
	At this stage, summarizing the above computations yields our first main theorem, Theorem~\ref{thm:BY-AF-large-boundary}.
	
	\medskip
	
	We then apply Theorem~\ref{thm:BY-AF-large-boundary} to several natural
	families of hypersurfaces approaching infinity. We begin with the
	coordinate spheres in the fixed asymptotically flat chart. In this
	case, the Euclidean geometry is exactly round, and the correction tensor
	$\mathfrak D_r$ appearing in \eqref{eq:BY-AF-main-expansion} vanishes
	identically. The shape-dependent term therefore disappears, and the
	Brown--York type mass converges to the ADM mass.
	
	\begin{corollary}\label{cor:BY-AF-coordinate-spheres}
		Under the hypotheses of Theorem~\ref{thm:BY-AF-large-boundary}, let
		$S_r=\{x\in\mathbb R^4:|x|=r\}$ be the coordinate sphere in the fixed
		asymptotically flat chart. Then, for all sufficiently large $r$,
		$m_{\BY}(S_r)$ is well-defined and
		\begin{align}
			\lim_{r\to\infty}m_{\BY}(S_r)
			=
			m_{\ADM}(g).
		\end{align}
	\end{corollary}
	
	\begin{proof}
		The family ${S_r}$ satisfies
		Assumption~\ref{ass:large-hypersurfaces}. Indeed, its rescaling by
		$r^{-1}$ is the unit sphere, and all three Euclidean principal
		curvatures are equal to $r^{-1}$.
		
		Let $\gamma_r=\gE|_{S_r}$. The Euclidean shape operator, its mean
		curvature, and the scalar curvature of $\gamma_r$ are
		\begin{align*}
			A_r^E
			=
			\frac{1}{r}\Id,
			\qquad
			H_r^E
			=
			\frac{3}{r},
			\qquad
			\Ric_{\gamma_r}^{\sharp}
			=
			\frac{2}{r^2}\Id,
			\qquad
			R_{\gamma_r}
			=
			\frac{6}{r^2}.
		\end{align*}
		Since $A_r^E$ is positive definite and satisfies the contracted
		Gauss equation, the uniqueness of the positive solution gives
		$A_0^+[\gamma_r]=A_r^E$. It follows that
		\begin{align*}
			B_{\gamma_r}
			&=
			\frac12R_{\gamma_r}\Id
			-
			\Ric_{\gamma_r}^{\sharp}
			=
			\frac{1}{r^2}\Id,
			\\
			P
			&=
			\frac12H_r^E B_{\gamma_r}^{-1}
			-
			B_{\gamma_r}^{-1}A_r^E
			=
			\frac r2\Id.
		\end{align*}
		After identifying $P$ with a symmetric $(2,0)$-tensor by means of
		$\gamma_r$, we therefore have
		$P^{\alpha\beta}=\frac r2\gamma_r^{\alpha\beta}$ and
		$p=\tr_{\gamma_r}P=\frac{3r}{2}$.
		
		Both $P$ and $p$ are parallel. Hence all derivative terms in
		$\mathcal E_{\gamma_r}(P)$ vanish. Using the constant-curvature identity on
		$S_r$, we obtain
		\begin{align*}
			P^{\gamma\delta}
			R(\gamma_r)_{\gamma}{}^{\alpha}{}_{\delta}{}^{\beta}
			&=
			\frac r2\Ric_{\gamma_r}^{\alpha\beta}
			=
			\frac1r\gamma_r^{\alpha\beta},
			\\
			-\frac12p\Ric_{\gamma_r}^{\alpha\beta}
			&=
			-\frac{3}{2r}\gamma_r^{\alpha\beta}.
		\end{align*}
		Consequently,
		\begin{align*}
			\mathcal E_{\gamma_r}(P)^{\alpha\beta}
			=
			-\frac{1}{2r}\gamma_r^{\alpha\beta}.
		\end{align*}
		On the other hand, raising both indices of the Euclidean second
		fundamental form gives
		$A_0^{\alpha\beta}=\frac1r\gamma_r^{\alpha\beta}$. Therefore,
		by \eqref{eq:BY-def-D}, the correction tensor $\mathfrak D_r$ vanishes identically. The conclusion follows directly from
		Theorem~\ref{thm:BY-AF-large-boundary}.
	\end{proof}
	
	\subsection{Nearly round hypersurfaces}\label{subsec:nearly-round-surface-BY-mass}

	The coordinate-sphere calculation shows that the correction tensor
	$\mathfrak D_a$ vanishes when the Euclidean geometry of the boundary is
	exactly round. It is therefore natural to ask whether the same
	large-boundary limit remains valid when the hypersurfaces are only
	approximately round. To formulate such a condition geometrically, we
	use the notion of nearly round surfaces introduced by Shi, Wang, and Wu
	\cite{ShiWangWu} as a model. Their definition concerns
	two-dimensional surfaces in three-dimensional asymptotically flat
	manifolds; below we introduce the corresponding conditions for
	three-dimensional hypersurfaces in the present four-dimensional
	setting.
	
	Fix a point $p\in M$ and set
	\begin{align*}
		r_p(x)
		:=
		d_g(p,x),
		\qquad
		s_a
		:=
		\inf_{x\in\Sigma_a}r_p(x).
	\end{align*}
	Let $\sigma_a=g|_{\Sigma_a}$. We denote by $K_a^g$, $H_a^g$, and
	$\mathring K_a^g$ the second fundamental form, mean curvature, and
	trace-free second fundamental form of $\Sigma_a$ with respect to $g$
	and the outward unit normal.
	
	\begin{definition}[Shi--Wang--Wu type nearly round hypersurfaces]
		\label{def:BY-AF-nearly-round}
		Let $\tau>0$. A family $\{\Sigma_a\}$ of smooth, closed
		hypersurfaces diffeomorphic to $S^3$ is said to be nearly round of
		rate $\tau$ if $s_a\to\infty$ and there exists a constant $C>0$,
		independent of $a$, such that the following conditions hold:
		\begin{enumerate}[label=\textup{(\roman*)}]
			\item The trace-free second fundamental form satisfies
			\begin{align}\label{eq:BY-nearly-round-tracefree}
				\left|\mathring K_a^g\right|_{\sigma_a}
				+
				s_a
				\left|
				\nabla^{\sigma_a}\mathring K_a^g
				\right|_{\sigma_a}
				\leq
				Cs_a^{-1-\tau}.
			\end{align}
			
			\item The inner and outer radii are uniformly comparable:
			\begin{align}\label{eq:BY-nearly-round-radius}
				\sup_{x\in\Sigma_a}r_p(x)
				\leq
				Cs_a+C.
			\end{align}
			
			\item The intrinsic diameter satisfies
			\begin{align}\label{eq:BY-nearly-round-diameter}
				\operatorname{diam}_{\sigma_a}(\Sigma_a)
				\leq
				Cs_a.
			\end{align}
			
			\item The induced volume satisfies
			\begin{align}\label{eq:BY-nearly-round-volume}
				\operatorname{Vol}_{\sigma_a}(\Sigma_a)
				\leq
				Cs_a^3.
			\end{align}
		\end{enumerate}
	\end{definition}
	
	The definition is intrinsic: it involves only the distance function,
	the induced metric, and the second fundamental form computed with
	respect to $g$. It is the dimensional analogue of the definition in
	\cite{ShiWangWu}, with the decay rate $\tau$ kept separate from the
	asymptotic-flatness order $q$. The power $s_a^3$ in
	\eqref{eq:BY-nearly-round-volume} replaces the area growth $r^2$ in
	the original two-dimensional setting.
	
	We first examine how these nearly round conditions are related to
	Assumption~\ref{ass:large-hypersurfaces}. They imply the required radius
	control and, together with the size assumptions, lead to uniform
	Euclidean convexity at the scale of the hypersurfaces. 
	
	\begin{lemma}\label{lem:BY-nearly-round-geometric-consequences}
		Let $\{\Sigma_a\}$ be a nearly round family of rate $\tau>0$.
		Assume, in addition, that each $\Sigma_a$ encloses $B_R$ in the
		fixed asymptotically flat chart. Then Assumption~\ref{ass:large-hypersurfaces} holds for
		all sufficiently large $a$.
	\end{lemma}
	
	\begin{proof}
		Let $\varrho(x):=|x|$ and $R_a:=\sup_{x\in\Sigma_a}\varrho(x)$, then $\rho_a=\inf_{x\in\Sigma_a}\varrho(x)$. Here $\varrho$ is the Euclidean radial function in the fixed
		asymptotically flat chart.
		
		We first compare the two radial functions $\varrho(x)$ and $d_g(p,x)$. Since
		$g-\gE=\mathcal O_2(|x|^{-q})$ with $q>1$, the length of a Euclidean radial
		ray computed using $g$ differs from its Euclidean length by a
		uniformly bounded amount. Together with the fact that $p$ is fixed,
		this gives
		\begin{align}\label{eq:two-radial-compare}
			\left|
			d_g(p,x)-\varrho(x)
			\right|
			\leq
			C
		\end{align}
		for all $x$ sufficiently far out in the end. Taking the infimum over
		$\Sigma_a$, we obtain
		\begin{align}\label{eq:BY-nearly-round-scale-comparison}
			|s_a-\rho_a|
			\leq
			C.
		\end{align}
		In particular, $s_a$ and $\rho_a$ are uniformly comparable for all sufficiently
		large $a$.
		
		Taking the supremum in \eqref{eq:two-radial-compare} and using the outer-radius condition \eqref{eq:BY-nearly-round-radius} in Definition~\ref{def:BY-AF-nearly-round}, we also get
		\begin{align*}
			R_a
			\leq
			\sup_{x\in\Sigma_a}d_g(p,x)+C
			\leq
			Cs_a+C
			\leq
			C\rho_a.
		\end{align*}
		Thus condition \textup{(1)} of
		Assumption~\ref{ass:large-hypersurfaces} holds.
		
		We next prove the curvature estimate in condition \textup{(2)} of Assumption~\ref{ass:large-hypersurfaces}. Set $\eta:=\min\{q,\tau\}$. The contracted Codazzi equation on $\Sigma_a$ gives
		\begin{align*}
			\nabla^{\sigma_a}H_a^g
			=
			\frac{3}{2}
			\left(
			\operatorname{div}_{\sigma_a}\mathring K_a^g
			-
			\Ric_g(\nu_g,\cdot)
			\right),
		\end{align*}
		Hence, by near roundness and asymptotic flatness \eqref{eq:BY-nearly-round-tracefree} and curvature decay on AF end,
		\begin{align}\label{eq:BY-nearly-round-gradient-H}
			\left|
			\nabla^{\sigma_a}H_a^g
			\right|_{\sigma_a}
			\leq
			C\rho_a^{-2-\eta}.
		\end{align}
		Indeed, the trace-free term is controlled by
		$s_a^{-2-\tau}$, while $|\Ric_g|=O(\rho_a^{-2-q})$. The diameter bound \eqref{eq:BY-nearly-round-diameter} in Definition~\ref{def:BY-AF-nearly-round}
		together with \eqref{eq:BY-nearly-round-scale-comparison} gives
		$\operatorname{diam}_{\sigma_a}(\Sigma_a)\leq C\rho_a$. Therefore
		\eqref{eq:BY-nearly-round-gradient-H} implies
		\begin{align}\label{eq:BY-nearly-round-H-oscillation}
			\operatorname{osc}_{\Sigma_a}H_a^g
			\leq
			C\rho_a^{-1-\eta}.
		\end{align}
		
		We now obtain the scale of $H_a^g$. Let $\Omega_a^E$ be the
		bounded Euclidean region enclosed by $\Sigma_a$. Since $\Sigma_a$
		encloses $B_R$ and $\rho_a=\inf_{\Sigma_a}|x|$, one has $B_{\rho_a}\subset\Omega_a^E\subset B_{R_a}$. Choose points $p_a^-,p_a^+\in\Sigma_a$ such that
		$\varrho(p_a^-)=\rho_a$ and $\varrho(p_a^+)=R_a$. At both points, the Euclidean outward unit normal to $\Sigma_a$ is
		$\partial_\varrho$. Since $g-\gE=O(\rho_a^{-q})$, the corresponding
		$g$-unit outward normal satisfies
		\begin{align*}
			\nu_g=\partial_\varrho+\mathcal O(\rho_a^{-q}),
			\qquad
			\nu_g(\varrho)=1+\mathcal O(\rho_a^{-q}).
		\end{align*}
		
		For tangent vectors $X,Y\in T\Sigma_a$, with the convention that
		large Euclidean spheres have positive second fundamental form, the restriction formula is
		\begin{align}\label{eq:BY-nearly-round-Hessian-restriction}
			\nabla_{\Sigma_a}^2\varrho(X,Y)
			= 
			\nabla_g^2\varrho(X,Y)
			-
			\nu_g(\varrho)K_a^g(X,Y).
		\end{align}
		Moreover,
		\begin{align}\label{eq:BY-nearly-round-radial-Hessian}
			\nabla_g^2\varrho
			=
			\frac1{\varrho}
			\bigl(g-d\varrho\otimes d\varrho\bigr)
			+
			O(\varrho^{-1-q})
		\end{align}
		in the asymptotically flat end.
		
		At $p_a^-$, the function $\varrho|_{\Sigma_a}$ has a minimum.
		Hence $\nabla_{\Sigma_a}^2\varrho\geq0$ there. Taking the trace of
		\eqref{eq:BY-nearly-round-Hessian-restriction} and using
		\eqref{eq:BY-nearly-round-radial-Hessian}, we obtain
		\begin{align*}
			H_a^g(p_a^-)
			\leq
			\frac{C}{\rho_a}.
		\end{align*}
		At $p_a^+$, the function $\varrho|_{\Sigma_a}$ has a maximum.
		Hence $\nabla_{\Sigma_a}^2\varrho\leq0$ there, and similarly
		\begin{align*}
			H_a^g(p_a^+)
			\geq
			\frac{c}{R_a}
			\geq
			\frac{c}{\rho_a}.
		\end{align*}
		Combining these estimates with \eqref{eq:BY-nearly-round-H-oscillation}, we
		get
		\begin{align}\label{eq:BY-nearly-round-H-scale}
			\frac{c}{\rho_a}
			\leq
			H_a^g
			\leq
			\frac{C}{\rho_a}
		\end{align}
		on $\Sigma_a$ for all sufficiently large $a$. 
		
		We are now in a position to prove condition \textup{(2)} of Assumption~\ref{ass:large-hypersurfaces}. The trace-free estimate \eqref{eq:BY-nearly-round-tracefree} in the nearly round condition, together with \eqref{eq:BY-nearly-round-gradient-H} and
		\eqref{eq:BY-nearly-round-H-scale} imply
		\begin{align}\label{eq:BY-nearly-round-physical-K-estimates}
			|K_a^g|_{\sigma_a}
			\leq
			C\rho_a^{-1},
			\qquad
			\left|
			\nabla^{\sigma_a}K_a^g
			\right|_{\sigma_a}
			\leq
			C\rho_a^{-2-\eta}.
		\end{align}
		With this estimate at hand, we can compare the physical and Euclidean second fundamental forms. Since $g - \gE = \mathcal{O}_2(|x|^{-q})$, the standard comparison of unit normals and Levi-Civita connections, combined with the estimate \eqref{eq:BY-nearly-round-physical-K-estimates} obtained above, yields
		\begin{align}\label{eq:BY-nearly-round-K-comparison}
			\begin{split}
				\left|
				K_a^E-K_a^g
				\right|_{\gamma_a}
				&\leq
				C\rho_a^{-1-q},
				\\
				\left|
				\nabla^{\gamma_a}K_a^E
				-
				\nabla^{\sigma_a}K_a^g
				\right|_{\gamma_a}
				&\leq
				C\rho_a^{-2-q}.
			\end{split}
		\end{align}
		It follows from
		\eqref{eq:BY-nearly-round-physical-K-estimates} and
		\eqref{eq:BY-nearly-round-K-comparison} that
		\begin{align}
			\left|
			\mathring K_a^E
			\right|_{\gamma_a}
			+
			\rho_a
			\left|
			\nabla^{\gamma_a}\mathring K_a^E
			\right|_{\gamma_a}
			\leq
			C\rho_a^{-1-\eta}.
		\end{align}
		Taking the trace in the comparison between $K_a^E$ and $K_a^g$, and using the previously obtained estimate \eqref{eq:BY-nearly-round-H-scale}, we obtain
		\begin{align}\label{eq:mean-curvature-bound}
			\frac{c}{\rho_a}
			\leq
			H_a^E
			\leq
			\frac{C}{\rho_a}
		\end{align}
		for all sufficiently large $a$, after possibly changing the constants $c$ and $C$. If $\lambda_{a,i}$ are the eigenvalues of
		$\mathring K_a^E$, then
		\begin{align*}
			\kappa_i^E
			=
			\frac{H_a^E}{3}
			+
			\lambda_{a,i},
			\qquad
			|\lambda_{a,i}|
			\leq
			C\rho_a^{-1-\eta}.
		\end{align*}
		Since $\eta>0$, the error term is lower order compared with
		$\rho_a^{-1}$. Hence, for all sufficiently large $a$,
		\begin{align*}
			\frac{c}{\rho_a}
			\leq
			\kappa_i^E
			\leq
			\frac{C}{\rho_a},
			\qquad
			i=1,2,3.
		\end{align*}
		This is condition \textup{(2)} of
		Assumption~\ref{ass:large-hypersurfaces}.
	\end{proof}
	
	We now estimate the correction term in
	Theorem~\ref{thm:BY-AF-large-boundary}. The following estimate is the
	main step.
	
	\begin{proposition}\label{prop:BY-nearly-round-correction}
		Let $\{\Sigma_a\}$ be a nearly round family of rate $\tau>0$.
		Assume, in addition, that, for all sufficiently large $a$, the
		hypersurface $\Sigma_a$ lies in the fixed asymptotically flat chart and
		encloses $B_R$. Set $\eta:=\min\{q,\tau\}$, where $q$ is the decay order
		of the asymptotically flat metric, and let $h:=g-\gE$ in the fixed
		asymptotically flat chart. Then there exists a constant $C$, independent
		of $a$, such that
		\begin{align}
			\label{eq:BY-nearly-round-correction-estimate}
			\left|
			\int_{\Sigma_a}
			\mathfrak D_a^{\alpha\beta}
			h_{\alpha\beta}\,d\mu_E
			\right|
			&\leq
			C\rho_a^{2-q-\eta}
		\end{align}
		for all sufficiently large $a$.
	\end{proposition}

	\begin{proof}
		By Lemma~\ref{lem:BY-nearly-round-geometric-consequences}, after
		discarding finitely many terms, the family satisfies Assumption~\ref{ass:large-hypersurfaces}. In particular,
		\begin{align}
			\label{eq:BY-nearly-round-Euclidean-umbilic-control}
			\left|\mathring K_a^E\right|_{\gamma_a}
			+
			\rho_a
			\left|\nabla^{\gamma_a}\mathring K_a^E\right|_{\gamma_a}
			\leq
			C\rho_a^{-1-\eta},
			\qquad
			\Vol_{\gE}(\Sigma_a)\leq C\rho_a^3,
		\end{align}
		and
		\begin{align}
			c\rho_a^{-1}\leq H_a^E\leq C\rho_a^{-1}.
		\end{align}
		For the Euclidean shape operator $A_a^E$ and mean curvature $H^E_a$, define
		\begin{align*}
			\varkappa_a
			:=
			\frac{1}{3\Vol_{\gE}(\Sigma_a)}
			\int_{\Sigma_a}H_a^E\,d\mu_E,
		\end{align*}
		then \eqref{eq:mean-curvature-bound} implies $c\rho_a^{-1}\leq \varkappa_a\leq C\rho_a^{-1}$. The Euclidean contracted Codazzi equation gives
		\begin{align*}
			\nabla^{\gamma_a}H_a^E
			=
			\frac32\operatorname{div}_{\gamma_a}\mathring K_a^E .
		\end{align*}
		Hence \eqref{eq:BY-nearly-round-Euclidean-umbilic-control}, together
		with the diameter bound \eqref{eq:BY-nearly-round-diameter}, implies
		\begin{align*}
			\left|\nabla^{\gamma_a}H_a^E\right|_{\gamma_a}
			\leq C\rho_a^{-2-\eta},
			\qquad
			\operatorname{osc}_{\Sigma_a}H_a^E
			\leq C\rho_a^{-1-\eta}.
		\end{align*}
		Since $3\varkappa_a$ is the average of $H_a^E$, we obtain
		\begin{align*}
			|H_a^E-3\varkappa_a|\leq C\rho_a^{-1-\eta},
		\end{align*}
		thus
		\begin{align}\label{eq:BY-nearly-round-A-close-umbilic}
			\begin{split}
				\left|A_a^E-\varkappa_a\Id\right|_{\gamma_a}
				&\leq
				\left|\mathring A_a^E\right|_{\gamma_a}
				+
				\frac13
				\left|H_a^E-3\varkappa_a\right|
				\left|\Id\right|_{\gamma_a}  \\
				&\leq
				C\rho_a^{-1-\eta}
				+
				C\rho_a^{-1-\eta}.
			\end{split}
		\end{align}
		
		Recall that on a Euclidean hypersurface
		\begin{align*}
			P_a
			=
			\frac{1}{J_a^E}
			\left(
			\frac12 H_a^E A_a^E-(A_a^E)^2
			\right),
			\qquad
			J_a^E=\det A_a^E .
		\end{align*}
		For the umbilic model $\bar A_a=\varkappa_a\Id$, the
		corresponding tensor is
		\begin{align*}
			\bar P_a=\frac{1}{2\varkappa_a}\operatorname{Id},
			\qquad
			\bar p_a=\tr_{\gamma_a}\bar P_a
			=
			\frac{3}{2\varkappa_a}.
		\end{align*}
		The map
		\begin{align*}
			A\longmapsto
			\frac{1}{\det A}
			\left(
			\frac12(\operatorname{tr}A)A-A^2
			\right)
		\end{align*}
		is smooth on the open set of positive definite self-adjoint endomorphisms. Using
		\eqref{eq:BY-nearly-round-A-close-umbilic} and the fact $c\rho_a^{-1}\leq \varkappa_a\leq C\rho_a^{-1}$, we arrive at
		\begin{align}
			\label{eq:BY-nearly-round-P-close-model}
			|P_a-\bar P_a|_{\gamma_a}
			+
			|p_a-\bar p_a|
			\leq
			C\rho_a^{1-\eta}.
		\end{align}
		Moreover, the Euclidean Gauss equation and
		\eqref{eq:BY-nearly-round-A-close-umbilic} imply
		\begin{align}
			\label{eq:BY-nearly-round-curvature-close-model}
			\left|
			\Rm_{\gamma_a}
			-
			\varkappa_a^2(\gamma_a\owedge\gamma_a)
			\right|_{\gamma_a}
			+
			\left|
			\Ric_{\gamma_a}
			-
			2\varkappa_a^2\gamma_a
			\right|_{\gamma_a}
			\leq
			C\rho_a^{-2-\eta},
		\end{align}
		where the normalization of $\gamma_a\owedge\gamma_a$ is chosen so that
		the model metric has sectional curvature $\varkappa_a^2$.
		
		Since $\bar P_a$ is parallel with respect to $\gamma_a$, the differential
		part of $\mathcal E_{\gamma_a}(\bar P_a)$ vanishes. For the umbilic model
		$\bar A_a=\varkappa_a\Id$ and the constant-curvature metric of sectional
		curvature $\varkappa_a^2$, one has
		\begin{align*}
			\mathcal E(\bar P_a)+\frac12\bar A_a=0.
		\end{align*}
		We decompose $\mathfrak D_a$ by separating the differential part from the
		remaining algebraic part:
		\begin{align*}
			\mathfrak D_a
			&=
			\mathcal L_{\gamma_a}(P_a-\bar P_a)
			+
			\mathcal A_a,
		\end{align*}
		where
		\begin{align*}
			\mathcal A_a^{\alpha\beta}
			:=
			-\frac12 p_a\Ric_{\gamma_a}^{\alpha\beta}
			+
			P_a^{\gamma\delta}
			R(\gamma_a)_{\gamma}{}^{\alpha}{}_{\delta}{}^{\beta}
			+
			\frac12(A_a^E)^{\alpha\beta}.
		\end{align*}
		Thus \eqref{eq:BY-nearly-round-A-close-umbilic},
		\eqref{eq:BY-nearly-round-P-close-model}, and
		\eqref{eq:BY-nearly-round-curvature-close-model} show that the algebraic
		part satisfies
		\begin{align*}
			|\mathcal A_a|_{\gamma_a}
			\leq
			C\rho_a^{-1-\eta}.
		\end{align*}
		
		It remains to estimate the term involving
		$\mathcal L_{\gamma_a}(P_a-\bar P_a)$. Since $\Sigma_a$ is closed, we may
		integrate by parts twice in the definition of $\mathcal L_{\gamma_a}$. Therefore,
		using \eqref{eq:BY-nearly-round-P-close-model} and volume estimate in \eqref{eq:BY-nearly-round-Euclidean-umbilic-control},
		\begin{align*}
			\left|
			\int_{\Sigma_a}
			\mathcal L_{\gamma_a}(P_a-\bar P_a)^{\alpha\beta}
			h_{\alpha\beta}\,d\mu_E
			\right|
			&
			\leq
			C\int_{\Sigma_a}
			|P_a-\bar P_a|_{\gamma_a}
			|(\nabla^{\gamma_a})^2h|_{\gamma_a}\,d\mu_E\\
			&
			\leq
			C\rho_a^{1-\eta}\rho_a^{-q-2}\Vol_{\gE}(\Sigma_a)
			\leq
			C\rho_a^{2-q-\eta}.
		\end{align*}
		Meanwhile, the algebraic part is estimated similarly:
		\begin{align*}
			\left|
			\int_{\Sigma_a}
			\mathcal A_a^{\alpha\beta}
			h_{\alpha\beta}\,d\mu_E
			\right|
			\leq
			C\rho_a^{-1-\eta}\rho_a^{-q}\Vol_{\gE}(\Sigma_a)
			\leq
			C\rho_a^{2-q-\eta}.
		\end{align*}
		Combining the last two estimates proves
		\eqref{eq:BY-nearly-round-correction-estimate}.
	\end{proof}
	
	Combining Theorem~\ref{thm:BY-AF-large-boundary} and Proposition~\ref{prop:BY-nearly-round-correction}, we establish our second main theorem, Theorem~\ref{thm:BY-AF-nearly-round}.
	
	\medskip
	
	We finish with a concrete class of nearly round hypersurfaces. The study of canonical constant-mean-curvature foliations at infinity was
	initiated, in the asymptotically flat setting, by Huisken--Yau
	\cite{HuiskenYau}, who constructed a unique foliation by stable CMC spheres near infinity for asymptotically flat three-manifolds of positive mass.
	Although their result concerns three-dimensional ambient manifolds, it
	provides the basic model for using geometrically distinguished CMC leaves as
	an exhaustion of an asymptotically flat end.  In four-dimensional Ricci-flat
	ALE geometry, Biquard--Hein use the canonical CMC foliation as a natural
	exhaustion in their study of renormalized volume \cite{BH}.  Thus, whenever a
	CMC foliation is available in the present setting and its leaves are
	asymptotically small normal graphs over coordinate spheres, it is natural to
	ask whether the Brown--York type mass converges to the ADM mass along those
	leaves.  The following corollary gives a concrete sufficient condition for
	this conclusion.

	\begin{corollary}
		Let $(M^4,g)$ satisfy the hypotheses of
		Theorem~\ref{thm:BY-AF-large-boundary}, and let $q>1$ be the
		asymptotic decay order in the fixed asymptotically flat chart.
		Let $r_a\to\infty$, and suppose that, for all sufficiently large $a$,
		\begin{align*}
			\Sigma_a
			&=
			\left\{
			r_a\bigl(1+\varphi_a(\omega)\bigr)\omega
			\;:\;
			\omega\in \Sph^3
			\right\}.
		\end{align*}
		where $\varphi_a\in C^\infty(\Sph^3)$ satisfies, for some $\beta>0$
		with $\beta>2-q$,
		\begin{align*}
			\|\varphi_a\|_{C^3(\Sph^3)}
			\leq
			Cr_a^{-\beta}.
		\end{align*}
		Then, after discarding finitely many terms, the family
		$\{\Sigma_a\}$ satisfies Assumption~\ref{ass:large-hypersurfaces} and
		is nearly round of rate $\eta=\min\{q,\beta\}$. Moreover,
		\begin{align*}
			\lim_{a\to\infty}m_{\BY}(\Sigma_a)
			=
			m_{\ADM}(g).
		\end{align*}
	\end{corollary}
	
	\begin{proof}
		Set $\eta := \min\{q, \beta\}$. We write the map $F_a:\Sph^3 \to \Sigma_a$ in the form $F_a(\omega) = r_a(1 + \varphi_a(\omega))\omega$. Since the $C^3$ norm of $\varphi_a$ converges to zero, we may discard finitely many terms and assume $1 + \varphi_a > 0$. Hence, $\Sigma_a$ is a smooth embedded radial graph. Moreover, both $\rho_a$ and $\sup_{\Sigma_a} |x|$ are equal to $r_a$ up to a relative error of $O(r_a^{-\beta})$. This implies $\rho_a \sim r_a$. In particular, the tail of the family lies in the fixed asymptotically flat chart and encloses $B_R$.
		
		The standard Euclidean radial-graph formulas give
		\begin{align*}
			r_a^{-2}F_a^*\gamma_a
			&=
			g_{\Sph^3}+\mathcal O_{C^2}(r_a^{-\beta}),\\
			r_a\,F_a^*A_a^E
			&=
			\Id+\mathcal O_{C^1}(r_a^{-\beta}).
		\end{align*}
		Thus
		\begin{align*}
			\kappa_i^E=r_a^{-1}(1+O(r_a^{-\beta}))
		\end{align*}
		for $i=1,2,3$. Since
		$\rho_a\sim r_a$, the Euclidean principal curvatures are uniformly comparable to
		$\rho_a^{-1}$. Therefore $\{\Sigma_a\}$ satisfies
		Assumption~\ref{ass:large-hypersurfaces}.
		
		The same radial-graph estimates imply
		\begin{align*}
			\left|\mathring K_a^E\right|_{\gamma_a}
			+
			r_a\left|\nabla^{\gamma_a}\mathring K_a^E\right|_{\gamma_a}
			\le
			Cr_a^{-1-\beta}.
		\end{align*}
		On the other hand, since $g-\gE=\mathcal O_2(|x|^{-q})$ and $\rho_a\sim r_a$, the induced
		metrics $\sigma_a$ and $\gamma_a$ are uniformly equivalent, and the standard comparison
		of unit normals, connections, and second fundamental forms gives
		\begin{align*}
			\left|\mathring K_a^g\right|_{\sigma_a}
			+
			r_a\left|\nabla^{\sigma_a}\mathring K_a^g\right|_{\sigma_a}
			\le
			C\bigl(r_a^{-1-\beta}+r_a^{-1-q}\bigr)
			\le
			Cr_a^{-1-\eta}.
		\end{align*}
		
		Let $s_a:=\inf_{\Sigma_a}d_g(p,\cdot)$. Since $q>1$, one has
		$d_g(p,x)=|x|+\mathcal O(1)$ on the asymptotically flat end. Hence $s_a\sim r_a$. Combining this
		with the radial-graph description and the uniform equivalence of $\sigma_a$ and $\gamma_a$,
		we obtain
		\begin{align*}
			\sup_{\Sigma_a}d_g(p,\cdot)
			&\le
			Cs_a+C,\\
			\diam_{\sigma_a}(\Sigma_a)
			&\le
			Cs_a,\\
			\Vol_{\sigma_a}(\Sigma_a)
			&\le
			Cs_a^3.
		\end{align*}
		Together with the preceding trace-free second fundamental form estimate, this proves that
		$\{\Sigma_a\}$ is nearly round of rate $\eta$.
		
		Finally, it is clear that $q+\eta>2$, since $q>1$ and $\beta>2-q$. The conclusion therefore follows from
		Theorem~\ref{thm:BY-AF-nearly-round}.
	\end{proof}

\end{document}